\documentclass[twoside,a4paper,reqno,11pt]{amsart}
\usepackage[top=28mm,right=30mm,bottom=28mm,left=30mm]{geometry}

\usepackage{amsfonts, amsmath, amssymb, mathrsfs, bm, latexsym, stmaryrd, array, hyperref, mathtools, bold-extra, amscd}

\usepackage[all]{xy}
\usepackage{xcolor}

\renewcommand{\a}{\alpha}
\renewcommand{\b}{\beta}

\newcommand{\e}{\varepsilon}
\renewcommand{\l}{\lambda}

\newcommand{\s}{\sigma}

\renewcommand{\O}{\Omega}

\newcommand{\C}{\mathcal{C}}

\newcommand{\la}{\langle}
\newcommand{\ra}{\rangle}

\renewcommand{\to}{\rightarrow}

\newcommand{\leqs}{\leqslant}
\newcommand{\geqs}{\geqslant}

\newcommand{\vs}{\vspace{2mm}}

\newcommand{\frat}{{\rm Frat}}
\newcommand{\End}{{\rm End}}
\newcommand{\diag}{{\rm diag}}
\newcommand{\soc}{{\rm soc}}

\newcommand{\Der}{{\rm Der}}

\makeatletter
\newcommand{\imod}[1]{\allowbreak\mkern4mu({\operator@font mod}\,\,#1)}
\makeatother

\newtheorem{theorem}{Theorem}
\newtheorem*{conj*}{Conjecture}

\newtheorem{thm}{Theorem}[section]
\newtheorem{prop}[thm]{Proposition}
\newtheorem{lem}[thm]{Lemma}
\newtheorem{cor}[thm]{Corollary}

\theoremstyle{definition}
\newtheorem{rem}[thm]{Remark}
\newtheorem{remk}{Remark}

\newtheorem*{def-non}{Definition}
\newtheorem{defn}[thm]{Definition}

\begin{document}

\author{Timothy C. Burness}
\thanks{Garonzi acknowledges the support of the Conselho Nacional de Desenvolvimento Cient\'ifico e Tecnol\'ogico (CNPq) - Grant numbers 302134/2018-2, 422202/2018-5.}
\address{T.C. Burness, School of Mathematics, University of Bristol, Bristol BS8 1UG, UK}
\email{t.burness@bristol.ac.uk}

\author{Martino Garonzi}
\address{M. Garonzi, Departamento de Matem\'{a}tica, Universidade de Bras\'{i}lia, Campus Universit\'{a}rio Darcy Ribeiro, Bras\'{i}lia-DF, 70910-900, Brazil}
\email{mgaronzi@gmail.com}

\author{Andrea Lucchini}
\address{A. Lucchini, Dipartimento di Matematica ``Tullio Levi-Civita”, Universit\`{a} di Padova, Via Trieste 63, 35131 Padova, Italy}
\email{lucchini@math.unipd.it}

\title[Finite groups, minimal bases and the intersection number]{Finite groups, minimal bases and \\ the intersection number}

\begin{abstract}
Let $G$ be a finite group and recall that the Frattini subgroup ${\rm Frat}(G)$ is the intersection of all the maximal subgroups of $G$. In this paper, we investigate the intersection number of $G$, denoted $\alpha(G)$, which is the minimal number of maximal subgroups whose intersection coincides with ${\rm Frat}(G)$. In earlier work, we studied $\alpha(G)$ in the special case where $G$ is simple and here we extend the analysis to almost simple groups. In particular, we prove that $\alpha(G) \leqslant 4$ for every almost simple group $G$, which is best possible. We also establish new results on the intersection number of arbitrary finite groups, obtaining upper bounds that are defined in terms of the chief factors of the group. Finally, for almost simple groups $G$ we present best possible bounds on a related invariant $\beta(G)$, which we call the base number of $G$. In this setting, $\beta(G)$ is the minimal base size of $G$ as we range over all faithful primitive actions of the group and we prove that the bound $\beta(G) \leqslant 4$ is optimal. Along the way, we study bases for the primitive action of the symmetric group $S_{ab}$ on the set of partitions of $[1,ab]$ into $a$ parts of size $b$, determining the exact base size for $a \geqslant b$. This extends earlier work of Benbenishty, Cohen and Niemeyer.
\end{abstract}

\date{\today}

\maketitle

\section{Introduction}\label{s:intro}

Let $G$ be a finite group and let $\mathcal{M}$ be the set of maximal subgroups of $G$. For $H \in \mathcal{M}$, let $H_G = \bigcap_{g \in G}H^g$ denote the core of $H$ and note that we may view $G/H_G$ as a primitive permutation group on the set $\O = G/H$ of cosets of $H$ in $G$. In this setting, a subset $B$ of $\O$ is a \emph{base} for $G/H_G$ if the pointwise stabiliser of $B$ in $G/H_G$ is trivial, and we define the \emph{base size} of $G$, denoted $b(G,H)$, to be the minimal size of a base. In other words, 
\[
b(G,H) = \min\{ |S| \,:\, S \subseteq G,\, \bigcap_{g \in S} H^g = H_G\}.
\]
Determining the base sizes of finite permutation groups (and primitive groups in particular) is a fundamental problem in permutation group theory, with a long history stretching back to the nineteenth century. 

Set $\mathcal{M}^* = \{ H \in \mathcal{M} \,:\, H_G = {\rm Frat}(G)\}$, where $\frat(G) = \bigcap_{H \in \mathcal{M}}H$ is the Frattini subgroup of $G$. In this paper, we are interested in the following invariant 
\[
\b(G) = \begin{cases}
\min\{b(G,H) \,:\, H \in \mathcal{M^*}\} & \text{if $\mathcal{M}^*\neq \emptyset$} \\
\infty&\text{otherwise,}
\end{cases}
\]
which we call the \emph{base number} of $G$. So if we assume $\mathcal{M}^*$ is non-empty, then $\b(G)$ is the smallest number of conjugate maximal subgroups whose intersection coincides with the Frattini subgroup of $G$. By relaxing the conjugacy condition, we obtain the \emph{intersection number} of $G$ (this terminology was introduced in \cite{Archer}):
\[
\a(G) = \min\{|\mathcal{T}| \,:\, \mathcal{T} \subseteq \mathcal{M},\, \bigcap_{H \in \mathcal{T}}H = \frat(G)\}.
\]
Clearly, we have $\a(G) \leqs \b(G)$. 

In an earlier paper \cite{BGL}, we studied these invariants in the setting where $G$ is a finite simple group, obtaining best possible bounds. More precisely, we proved that $\a(G) \leqs 3$ (with equality for infinitely many simple groups) and $\b(G) \leqs 4$ (with equality if and only if $G$ is isomorphic to the unitary group ${\rm U}_{4}(2)$). This extended earlier work of Garonzi and Lucchini \cite{GL}, who determined the exact intersection number of the alternating groups. Results on $\a(G)$ for some families of insoluble groups are presented in \cite{Archer} by Archer et al. For example, \cite[Theorem 3.3]{Archer} gives a formula for the intersection number of every finite nilpotent group, and there are results for dihedral, generalised quaternion and symmetric groups in \cite[Section 4]{Archer}. In particular, \cite[Proposition 4.3]{Archer} gives the bound 
\[
\a(S_n) \leqs \left\lfloor \frac{n+8}{4} \right\rfloor,
\]
which we improve to $\a(S_n) \leqs 3$ in Theorem \ref{t:main1} below (the latter bound is best possible).

\vs

In this paper, we extend some of this earlier work in several directions. Our first main result generalises the bounds on $\a(G)$ and $\b(G)$ in \cite[Theorem 1]{BGL} from simple groups to almost simple groups, providing best possible bounds for both invariants. 

\begin{theorem}\label{t:main1}
Let $G$ be a finite almost simple group with socle $G_0$. 
\begin{itemize}\addtolength{\itemsep}{0.2\baselineskip}
\item[{\rm (i)}]  We have $\a(G) \leqs 4$, with equality if and only if $G \cong {\rm U}_{4}(2).2$.
\item[{\rm (ii)}]  We have $\b(G) \leqs 4$, with equality if and only if $G \cong S_6$ or $G_0 \cong {\rm U}_{4}(2)$.
\end{itemize}
\end{theorem}

\begin{remk}\label{r:1}
As an immediate corollary, it follows that $\b(G) - \a(G) \leqs 1$ for every almost simple group $G$, which extends \cite[Corollary 2(iii)]{BGL}. It is also worth noting that there are infinitely many almost simple groups with $\a(G) = \b(G) = 3$. For example, every alternating group of the form $A_{2p}$ has this property, where $p \geqs 17$ is a prime and $2p-1$ is not a prime power (see \cite[Theorem 1(i)]{BGL}).
\end{remk}

A key ingredient in the proof of Theorem \ref{t:main1} for symmetric groups is Theorem \ref{t:main2} below, which we anticipate will be of independent interest. 

Let $G = S_n$ with $n \geqs 5$ and let $H \ne A_n$ be a maximal subgroup of $G$, so we may view $G$ as a primitive permutation group on the cosets of $H$. Consider the action of $H$ on $[1,n] = \{1, \ldots, n\}$. If $H$ acts primitively, then $b(G,H)$ is determined precisely by Burness, Guralnick and Saxl in \cite{BGS} (for example, it turns out that $b(G,H) = 2$ if $n>12$). If $H$ is intransitive, then the best known bounds are due to Halasi \cite{Halasi}, who gives the exact base size in many (but not all) cases. Now assume $H$ is imprimitive, so $n = ab$ and $H = S_b \wr S_a$ for integers $a,b \geqs 2$. Here the best existing bounds in the literature are due to Benbenishty, Cohen and Niemeyer \cite{BCN} (see Theorem \ref{t:bcn}). In particular, if $a \geqs b \geqs 3$, then \cite[Theorem 4]{BCN} gives $b(G,H) \leqs 6$. By applying work of James \cite{James}, we are able to determine the exact base size for this action whenever $a \geqs b$.

\begin{theorem}\label{t:main2}
Let $G = S_n$ and $H = S_b \wr S_a$, where $n=ab$, $a \geqs b \geqs 2$ and $(a,b) \ne (2,2)$. Then 
\[
b(G,H) = \left\{
\begin{array}{ll}
4 & \mbox{if $(a,b) = (3,2)$} \\
2 & \mbox{if $b \geqs 3$ and $a \geqs \max\{b+3,8\}$} \\
3 & \mbox{otherwise.}
\end{array}\right.
\]
\end{theorem}

The proof of Theorem \ref{t:main2} is constructive in the sense that we present a base of minimal size  in each case. We also obtain a corresponding result for alternating groups (see Remark \ref{r:partan}).

\vs

Finally, we turn our attention to the intersection numbers of arbitrary finite groups. In the following we write $\l(G)$ for the chief length of $G$ (so $\l(G)$ is the number of factors in a chief series for $G$) and $\delta(G)$ denotes the number of \emph{non-Frattini} chief factors of $G$, which is independent of the choice of chief series (recall that a chief factor $H/K$ of $G$ is \emph{Frattini} if it is contained in ${\rm Frat}(G/K)$). It is straightforward to show that if $G$ is nilpotent, then $\alpha(G)=\lambda(G/\frat(G))=\delta(G)$ (see \cite[Theorem 3.3]{Archer} for example). We extend the analysis to soluble groups.

\begin{theorem}\label{t:main3}
If $G$ is a finite soluble group, then $\a(G) \leqs \l(G)$. Moreover, if the derived subgroup of $G$ is nilpotent, then $\a(G) \leqs \delta(G)$.
\end{theorem}

\begin{remk}\label{r:2}
It is worth noting that the difference $\alpha(G)-\delta(G)$ can be arbitrarily large for soluble groups. For example, if $k$ is a positive integer and we take $\Gamma_k$ to be the finite soluble group defined at the end of \cite[Section 8]{BGL}, then $\delta(\Gamma_k)=5k$ and $\alpha(\Gamma_k)\geqs 6k.$ However, we can still bound $\alpha(G)$ in terms of $\delta(G)$. For example, as a corollary of 
Theorem \ref{t:main4} below, we deduce that $\alpha(G)\leqs 4\delta(G)$ for every finite soluble group $G.$
\end{remk}

Finally, we present a general bound on the intersection number of an arbitrary finite group $G$. In order to state Theorem \ref{t:main4} below, we need to introduce some notation. Let $\mathcal B_{\text{ab}}$ (respectively $\mathcal B_{\text{nonab}}$) be the set of non-Frattini chief factors of $G$ that are $G$-equivalent to some abelian (respectively, non-abelian) minimal normal subgroup of $G/\frat(G)$ (see Definition \ref{d:equiv} for the definition of \emph{$G$-equivalent}). In addition, let $\delta_G(A)$ be the number of non-Frattini chief factors in a chief series of $G$ which are $G$-equivalent to $A$ (this does not depend on the choice of chief series), and if $A$ is non-abelian, let $n_A$ be the number of composition factors of $A$.

\begin{theorem}\label{t:main4}
If $G$ is a finite group, then
\[
\alpha(G)\leqs \sum_{A \in \mathcal B_{{\rm ab}}}\delta_G(A)+\!\!\sum_{A \in \mathcal B_{{\rm nonab}}}\!\!\max\{4,\delta_G(A)\}+\!\!
	\sum_{A \in \mathcal B_{{\rm ab}}}\dim_{\End_G(A)}A+\!\!\sum_{A \in \mathcal B_{{\rm nonab}}}\left\lfloor\frac{3n_A-1}{2}\right\rfloor.
	\]
In particular, if $G$ is soluble then
	\begin{equation}\label{e:sol}
	\alpha(G)\leqs \sum_{A \in \mathcal B_{\text{ab}}}(\delta_G(A)+3).
	\end{equation}
\end{theorem}

\begin{remk}\label{r:3}
Let us observe that the bound in \eqref{e:sol} for soluble groups is best possible. For example, let $t$ be a positive integer and let $R < S_{4^t}$ be the iterated wreath product of $t$ copies of $S_4$. The wreath product $H={\rm GL}_{2}(3) \wr R$ admits a faithful irreducible $H$-module $A$ of order $n=9^{4^t}$. Let $\delta$ be a positive integer and consider the semidirect product $G=A^\delta \rtimes H.$ Then $\mathcal B_{\text{ab}}=\{A\}$, $\delta_G(A)=\delta$ and $|G|=n^{\delta+c-1}24^{-1/3}$, where $c=1+\log_9(48\cdot 24^{1/3}) \sim 3.244$ is the P\'{a}lfy-Wolf constant (see the final part of \cite[Section 1]{palfy}, for example). Since $|G:H| \leqs n$ for every maximal subgroup $H$ of $G$, it follows that
\[
\alpha(G)\geqs \lceil \log_n|G|\rceil=\lceil \delta+c-1-\log_n(24)/3\rceil\geqs \delta+3
\]
if $c-1-\log_n(24)/3>2$, which holds for $n$ sufficiently large. We conclude that if $n \gg 0$, then the bound in \eqref{e:sol} is sharp. 
\end{remk}

Notice that the general upper bound in Theorem \ref{t:main4} involves the composition length of each $A \in \mathcal{B}_{{\rm nonab}}$. It remains an open problem to determine if it is possible to bound $\alpha(G)$ only in terms of $\delta(G)$ when $G$ is insoluble.

\section{Symmetric groups acting on partitions}\label{s:sym}

We begin by proving Theorem \ref{t:main2}, which will then be used in the proof of Theorem \ref{t:main1} in the next section. So let $G = S_n$ and consider the imprimitive subgroup $H = S_b \wr S_a$, where $n=ab$ and $a,b \geqs 2$. If $n=4$ then $H$ contains a nontrivial normal subgroup of $G$, so we will assume $(a,b) \ne (2,2)$. We may then view $G$ as a primitive permutation group on the set of cosets of $H$ in $G$, which we can identify with the set $\O$ of partitions of $[1,n]=\{1, \ldots, n\}$ into $a$ sets of size $b$ (we refer to the sets in such a partition as \emph{blocks}).

There are several results in the literature on the base size $b(G,H)$ for this action. For example, a theorem of Liebeck \cite{Lie} states that if $b \geqs 3$ then 
\[
b(G,H) \leqs (a-1)(b-1)+2.
\]
Asymptotically best possible bounds on $b(G,H)$ were determined more recently by Benbenishty, Cohen and Niemeyer. The following result is \cite[Theorem 4]{BCN}.

\begin{thm}\label{t:bcn}
Let $G = S_{n}$ and $H = S_b \wr S_a$, where $n=ab$ with $a \geqs 2$ and $b \geqs 3$.
\begin{itemize}\addtolength{\itemsep}{0.2\baselineskip}
\item[{\rm (i)}] If $a \geqs b$, then $2 \leqs b(G,H) \leqs 6$.
\item[{\rm (ii)}] If $a<b$, then $\lceil \log_a b \rceil \leqs b(G,H) \leqs \lceil \log_a b \rceil+3$.
\end{itemize}
\end{thm}

\begin{rem}\label{r:two}
As explained in \cite[p.1581]{BCN}, the lower bound in part (ii) follows immediately from the observation that in any collection of fewer than $\lceil \log_a b \rceil$ partitions in $\O$, there are at least two points that appear in the same block in every partition in the collection, so $G$ contains a transposition fixing every partition. It is also worth noting that the upper bound in (ii) is best possible. For example, if $(a,b) = (2,4)$ then $b(G,H) = 5 = \lceil \log_a b \rceil+3$.
\end{rem}

Here we are interested in extending part (i) of Theorem \ref{t:bcn} by determining the exact base size whenever $a \geqs b$ (and we also handle the case $b=2$); this is the content of Theorem \ref{t:main2}. A key tool to do this is the following result of James \cite[Theorem 1.2]{James}, which determines the cases with $b(G,H) = 2$.

\begin{thm}\label{t:james}
Let $G = S_{n}$ and $H = S_b \wr S_a$, where $n=ab$ with $a,b \geqs 2$ and $(a,b) \ne (2,2)$. Then $b(G,H) = 2$ if and only if $b \geqs 3$ and $a \geqs \max\{8,b+3\}$.
\end{thm}

If $b=2$ and $a \geqs 4$ then $b(G,H) = 3$ (see \cite[Remark 1.6]{BGS}) and it is easy to check that $b(G,H) = 4$ if $(a,b) = (3,2)$. Now assume $b \geqs 3$. If $a \geqs b+3$ then Theorem \ref{t:james} gives $b(G,H) = 2$ if and only if $a \geqs 8$; if $a \leqs 7$ then $(a,b)$ is one of $(6,3)$, $(7,3)$ or $(7,4)$ and in each case one checks that $b(G,H) = 3$ (for example, with the aid of {\sc Magma} it is easy to identify two elements $x,y \in G$ such that $H \cap H^x \cap H^y = 1$, which implies that $b(G,H) \leqs 3$ and therefore equality holds by Theorem \ref{t:james}). We have now established the following result.

\begin{prop}\label{p:sym1}
Let $G = S_n$ and $H = S_b \wr S_a$, where $n=ab$, $a \geqs b \geqs 2$ and $(a,b) \ne (2,2)$. If $b=2$ or $a \geqs b+3$, then  
\[
b(G,H) = \left\{
\begin{array}{ll}
4 & \mbox{if $(a,b) = (3,2)$} \\
3 & \mbox{if $a \geqs 4$ and $b=2$, or $(a,b) = (6,3)$, $(7,3)$ or $(7,4)$} \\
2 & \mbox{otherwise.}
\end{array}\right.
\]
\end{prop}

Therefore, to complete the proof of Theorem \ref{t:main2}, we may assume $3 \leqs b \leqs a \leqs b+2$. In this situation, by combining Theorems \ref{t:bcn} and \ref{t:james}, we have
\[
3 \leqs b(G,H) \leqs 6
\]
and our goal is to prove that $b(G,H) = 3$. As in \cite{BCN}, our approach is constructive and we will exhibit an explicit base of size $3$ in every case. We divide the analysis into three cases: $a=b+2$, $a=b+1$ and $a=b$.

\subsection{The case $a=b+2$}\label{ss:plus2}

\begin{prop}\label{p:plus2}
Let $G = S_n$ and $H = S_{b} \wr S_a$, where $n=ab$, $a = b+2$ and $b \geqs 3$. Then $b(G,H) = 3$.
\end{prop}

\begin{proof}
The case $b=3$ can be verified using {\sc Magma}, so for the remainder we will assume $b \geqs 4$. We begin by identifying $G$ with ${\rm Sym}(X)$, where  
\[
X=\{(i,j) \in \mathbb{Z}/a\mathbb{Z} \times \mathbb{Z}/a\mathbb{Z}\,:\,  i-j \neq \pm 1\},
\]
and then we identify $G/H$ with the set $\O$ of partitions of $X$ into $a$ parts of size $b$. In view of Theorem \ref{t:james}, it suffices to identify three partitions in $\O$ whose pointwise stabiliser in $G$ is trivial.

With this goal in mind, consider the following subsets of $X$
\begin{align*}
B_i & = \{(x,y) \in X\, :\, x=i\}, \;\; i = 0,1,\ldots,a-1 \\
C_i & = \{(x,y) \in X\, :\, y=i\}, \;\; i = 0,1,\ldots,a-1 \\
D_0 & = (B_0 \setminus \{(0,2)\}) \cup \{(1,3)\} = \{(1,3),(0,3),(0,4),\ldots,(0,a-2),(0,0)\} \\
D_1 & = (B_1 \setminus \{(1,3)\}) \cup \{(0,2)\} = \{(0,2),(1,4),(1,5),\ldots,(1,a-1),(1,1)\} \\
D_i & = B_i, \;\; i = 2,3,\ldots,a-3 \\
D_{a-2} & = \{(a-1,1),(a-1,2),(a-2,2),(a-2,3),\ldots,(a-2,a-4),(a-2,a-2)\} \\
D_{a-1} & = \{(a-2,0),(a-2,1),(a-1,3),(a-1,4),\ldots,(a-1,a-3),(a-1,a-1)\}
\end{align*}
and observe that  
\[
\mathcal{B} = \{B_0,\ldots,B_{a-1}\}, \;\; \mathcal{C} = \{C_0,\ldots,C_{a-1}\}, \;\;
\mathcal{D} = \{D_0,\ldots,D_{a-1}\}
\]
are all partitions in $\O$. 

Suppose $g \in \mbox{Sym}(X)$ stabilises $\mathcal{B}$, $\mathcal{C}$ and $\mathcal{D}$. We claim that $g=1$, which implies that $\{\mathcal{B},\mathcal{C},\mathcal{D}\}$ is a base for $G$. Since $(i,j) \in X$ is the unique element in the intersection $B_i \cap C_j$, it suffices to show that $g$ fixes (setwise) each block in $\mathcal{B}$ and $\mathcal{C}$.

First observe that $D_1 \setminus B_1=\{(0,2)\}$ and $D_0\setminus B_0=\{(1,3)\}$, which means that 
\[
|g(D_1\setminus B_1)| = |g(D_0\setminus B_0)| = 1.
\] 
In other words, $|g(D_1)\setminus g(B_1)| = |g(D_0)\setminus g(B_0)|=1$. Since  $|D_k\setminus B_i|=1$ if and only if $(k,i) \in \{(0,0), (1,1)\}$, and $|B_i\setminus D_k|=1$ if and only if $(k,i) \in \{(0,0),(1,1)\}$, the fact that $a \geqs 6$ implies that $\{g(D_0), g(D_1)\} = \{D_0, D_1\}$ and $\{g(B_0), g(B_1)\} = \{B_0, B_1\}$, with $g(B_0)=B_0$ if and only if $g(D_0)=D_0$.

Suppose $g(B_0)=B_1$, so $g(B_1)=B_0$, $g(D_0)=D_1$ and $g(D_1)=D_0$. Since $D_1\setminus B_1=\{(0,2)\}$ and $D_0\setminus B_0=\{(1,3)\}$, we deduce that $g(0,2)=(1,3)$ and $g(1,3)=(0,2)$. This implies that $g(C_2)=C_3$ and $g(C_3)=C_2$. However $D_0 \cap C_3 = \{(1,3),(0,3)\}$ and $g(D_0 \cap C_3) = D_1 \cap C_2 = \{(0,2)\}$ have different sizes, so we have reached a contradiction. Therefore, $g$ fixes the blocks $B_0, B_1, D_0$ and $D_1$. 

Since $D_1\setminus B_1=\{(0,2)\}$ and $D_0\setminus B_0=\{(1,3)\}$ we deduce that $g(0,2)=(0,2)$ and $g(1,3)=(1,3)$, so $g(C_j)=C_j$ for $j=2,3$. Now $D_0 \cap C_1 = \emptyset$ and by applying $g$ we obtain $D_0 \cap g(C_1) = \emptyset$. Since $g(C_2)=C_2$, this forces $g(C_1) \in \{C_1,C_{a-1}\}$. Since $a \geqs 6$, we have $|D_k \cap B_i|=2$ only if $k \in \{a-2,a-1\}$ (independently of $i$), therefore $\{D_{a-1},D_{a-2}\}$ is stabilised by $g$. Now $C_1 \cap D_{a-1}$ and $C_1 \cap D_{a-2}$ are nonempty while $C_{a-1} \cap D_{a-2} = \emptyset$, therefore $g(C_1) \neq C_{a-1}$, implying $g(C_1)=C_1$.

Visibly, we have $B_i \cap C_j = \emptyset$ if and only if $j = i \pm 1$. So if $h \in \mbox{Sym}(X)$ and $h(B_i)=B_i$ then $h(B_i \cap C_{i \pm 1}) = B_i \cap h(C_{i \pm 1}) = \emptyset$ and thus $\{h(C_{i-1}), h(C_{i+1})\} = \{C_{i-1},C_{i+1}\}$.  Similarly, if $h(C_j)=C_j$ then $\{h(B_{j-1}), h(B_{j+1})\} = \{B_{j-1},B_{j+1}\}$.

We now repeatedly apply this observation, given the constraints on $g$ we have already obtained. Firstly, since $g$ fixes $C_1$ it stabilises $\{B_0,B_2\}$. But we know that $g$ fixes $B_0$, so it must also fix $B_2$ and hence it stabilises $\{C_1,C_3\}$. Since it fixes $C_1$, it also fixes $C_3$, so it stabilises $\{B_2,B_4\}$. In turn, since $g$ fixes $B_2$, it must also fix $B_4$. By continuing the argument in this way, we deduce that $g$ fixes $B_i$ for every even $i$ and $C_j$ for every odd $j$. 

Similarly, $g$ stabilises $\{B_1,B_3\}$ since it fixes $C_2$. As before, since we already know that it fixes $B_1$, it must also fix $B_3$ and thus $g$ stabilises $\{C_2,C_4\}$. It follows that $g$ fixes $C_4$, so it stabilises $\{B_3,B_5\}$ and we deduce that it fixes $B_5$. Once again, proceeding in this way we find that $g$ fixes $B_i$ for every odd $i$ and $C_j$ for every even $j$. 

We have now shown that $g$ fixes each block in $\mathcal{B}$ and $\mathcal{C}$. As previously noted, if $i \ne j \pm 1$ then $B_i \cap C_j = \{ (i,j) \}$ and thus $g$ fixes every element in $X$. Therefore, $g=1$ and the result follows.
\end{proof}

\subsection{The case $a=b+1$}\label{ss:plus1}

\begin{prop}\label{p:plus1}
Let $G = S_n$ and $H = S_{b} \wr S_a$, where $n=ab$, $a = b+1$ and $b \geqs 3$. Then $b(G,H) = 3$.
\end{prop}

\begin{proof}
Here we identify $G$ with ${\rm Sym}(X)$, where 
\[
X=\{(i,j) \in \mathbb{Z}/a\mathbb{Z} \times \mathbb{Z}/a\mathbb{Z}\, :\, j-i \neq  1\}
\]
and we define $\O$ to be the set of partitions of $X$ into $a$ parts of size $b$. The case $b=3$ can be verified using {\sc Magma}, so we may assume $b \geqs 4$.

We define the partitions $\mathcal{B} = \{B_0, \ldots, B_{a-1}\}$ and $\mathcal{C} = \{C_0, \ldots, C_{a-1}\}$ in $\O$ as in the proof of Proposition \ref{p:plus2}, so
\[
B_i = \{(x,y) \in X\, :\, x=i\},\;\; C_i= \{(x,y) \in X\,:\, y=i\},\;\; i = 0,1,\ldots, a-1.
\]
We also define a third partition $\mathcal{D} = \{D_0, \ldots, D_{a-1}\}$ in $\O$. 
For $i \geqs 0$ we define 
\[
Y_i = \{(2i,2),(2i,4),\ldots, (2i,2i+2)\},\;\; Z_i = \{(2i+1,3),(2i+1,5), \ldots, (2i+1,2i+3)\}
\]
and we take
\begin{align*}
D_{2i} & = (B_{2i}\setminus Y_i) \cup Z_i, \;\; i = 0,1, \ldots, \lfloor a/2 \rfloor -1 \\
D_{2i+1} & = (B_{2i+1}\setminus Z_i) \cup Y_i, \;\; i =  0,1, \ldots, \lfloor a/2 \rfloor -1 \\
D_{a-1} & = B_{a-1} \mbox{ if $a$ is odd}.
\end{align*}

Suppose $g \in \mbox{Sym}(X)$ stabilises $\mathcal{B}$, $\mathcal{C}$ and $\mathcal{D}$. We claim that $g=1$, which implies that $b(G,H) \leqs 3$. Recall that this gives the desired result because $b(G,H) \geqs 3$ by \cite[Theorem 1.2]{James}. We proceed as in the proof of the previous proposition, noting that it suffices to show that $g$ fixes each block in $\mathcal{B}$ and $\mathcal{C}$. In order to reach this conclusion, it will be useful to observe that $B_i \cap C_j=\emptyset$ if and only if $j=i+1$, hence if $g(B_i)=B_j$ then $g(C_{i+1})=C_{j+1}.$

First observe that $D_1\setminus B_1=\{(0,2)\}$ and $D_0\setminus B_0=\{(1,3)\}$, so 
\[
|g(D_1)\setminus g(B_1)| = |g(D_1\setminus B_1)| =  |g(D_0\setminus B_0)| = |g(D_0)\setminus g(B_0)|=1.
\]
Since $a \geqs 5$, $|D_k\setminus B_i| = 1$ if and only if $(k,i) \in \{(0,0), (1,1)\}$, and $|B_i\setminus D_k|=1$ if and only if $(k,i) \in \{(0,0),(1,1)\}$. It follows that 
$\{g(D_0), g(D_1)\} = \{D_0, D_1\}$ and $\{g(B_0), g(B_1)\} = \{B_0, B_1\}$. Moreover, $g(B_0)=B_0$ if and only if $g(D_0)=D_0$.

Suppose $g(B_0)=B_1$. Then $g(B_1)=B_0$, $g(D_0)=D_1$ and $g(D_1)=D_0$. Since $D_1\setminus B_1=\{(0,2)\}$ and $D_0\setminus B_0=\{(1,3)\}$, it follows that $g(0,2)=(1,3)$ and $g(1,3)=(0,2)$. Therefore $g(C_2)=C_3$ and $g(C_3)=C_2$. However $D_0 \cap C_3 = \{(1,3),(0,3)\}$ and $g(D_0 \cap C_3) = D_1 \cap C_2 = \{(0,2)\}$ have different sizes, which is a contradiction.

We have now shown that $g(B_0)=B_0$, so $g$ also fixes the blocks $B_1, D_0$ and $D_1$. We also observe that $g$ fixes $C_1$ and $C_2$.

We now argue inductively to complete the proof. Suppose we have proved that $g(B_j)=B_j$ and $g(C_k)=C_k$ for all $0\leqs j\leqs 2i-1$ and $1\leqs k\leqs 2i$ for some $i$ in the range $1 \leqs i \leqs (a-3)/2$. 

First assume $i\neq (a-3)/2$ when $a$ is odd, and $i \neq a/2-2$ when $a$ is even. Then we have $|D_k\setminus B_j| = i+1$ if and only if $(k,j) \in \{(2i,2i),(2i+1,2i+1)\}$, and $|B_j\setminus D_k| = i+1$ if and only if $(k,j) \in \{(2i,2i),(2i+1,2i+1)\}$. Therefore, $\{g(B_{2i}), g(B_{2i+1})\} = \{B_{2i}, B_{2i+1}\}$ and $\{g(D_{2i}), g(D_{2i+1})\} = \{D_{2i}, D_{2i+1}\}$, with $g(B_{2i})=B_{2i}$ if and only if $g(D_{2i})=D_{2i}$. Since $D_{2i} \cap C_{2i} = \emptyset$, $D_{2i+1} \cap C_{2i} \neq \emptyset$ and $g(C_{2i})=C_{2i}$, we deduce that $g(D_{2i}) = D_{2i}$ and $g(D_{2i+1}) = D_{2i+1}$, which in turn implies that $g$ fixes $B_{2i}$, $B_{2i+1}$, $C_{2i+1}$ and $C_{2i+2}.$

Now assume $a$ is odd and $i=(a-3)/2.$ Here $|D_k\setminus B_j| = i+1$ if and only if 
\[
(k,j) \in \{(a-3,a-3),(a-2,a-2),(a-3,a-2),(a-2,a-3)\}
\]
so
$g(\{D_{a-3},D_{a-2}\})=\{D_{a-3},D_{a-2}\}$ and
$g(\{B_{a-3},B_{a-2}\})=\{B_{a-3},B_{a-2}\}$. As before,
since $D_{a-3} \cap C_{a-3} = \emptyset$, $D_{a-2} \cap C_{a-3} \neq \emptyset$ and $g(C_{a-3})=C_{a-3}$, we deduce that $g(D_{a-3}) = D_{a-3}$ and $g(D_{a-2}) = D_{a-2}$. So either $g(B_{a-3})=B_{a-3}$ and $g(B_{a-2})=B_{a-2}$, or
\[
g(\{(a-3,0),(a-3,1),(a-3,3),\dots,(a-3,a-4)\})=g(B_{a-3}\cap D_{a-3})=B_{a-2}\cap D_{a-3},
\]
which is equal to $\{(a-2,0),(a-2,3),(a-2,5), \ldots,(a-2,a-2)\}$. But here the second possibility is incompatible with the fact that $|B_{a-3} \cap D_{a-3} \cap C_1|=1$, $B_{a-2} \cap D_{a-3} \cap C_1 = \emptyset$ and $g(C_1)=C_1$, whence $g(B_{a-3})=B_{a-3}$ and $g(B_{a-2})=B_{a-2}$, and thus $g(C_{a-2}) = C_{a-2}$ and $g(C_{a-1}) = C_{a-1}$.

Now suppose $a$ is even and $i=a/2-2$. Here $|D_k\setminus B_j| = i+1$ if and only if $|B_j\setminus D_k| = i+1$, and this happens if and only if
\[
(k,j) \in \{(a-4,a-4),(a-3,a-3),(a-2,a-1),(a-1,a-2)\}, 
\]
so $g$ stabilises the sets
\[
\{B_{a-4},B_{a-3},B_{a-2},B_{a-1}\},\;\; \{D_{a-4},D_{a-3},D_{a-2},D_{a-1}\}.
\]
Moreover, in order to prove that $g$ fixes $B_{a-4}$, $B_{a-3}$, $B_{a-2}$ and $B_{a-1}$, it suffices to show that $g$ fixes $D_{a-4}$, $D_{a-3}$, $D_{a-2}$ and $D_{a-1}$. Observe that $D_{a-4} \cap C_{a-4} = D_{a-2} \cap C_{a-4} = \emptyset$, while $D_{a-3} \cap C_{a-4}$ and $D_{a-1} \cap C_{a-4}$ are non-empty. Since $g$ fixes $C_{a-4}$, we deduce that $g$ stabilises $\{D_{a-4},D_{a-2}\}$ and $\{D_{a-3},D_{a-1}\}$. Since $D_{a-4}$ and $D_{a-3}$ intersect precisely $a/2+1$ blocks in $\mathcal{C}$, whereas $D_{a-2}$ and $D_{a-1}$ intersect $a/2$ such blocks, we deduce that $g$ fixes $D_{a-4}$, $D_{a-3}$, $D_{a-2}$ and $D_{a-1}$.

So by induction, it follows that $g$ fixes every block in $\mathcal{B}$ and $\mathcal{C}$. Since each $(i,j) \in X$ is the unique element of the intersection $B_i \cap C_j$, we conclude that $g=1$.
\end{proof}

\subsection{The case $a=b$}\label{ss:plus0} 

\begin{prop}\label{p:plus0}
Let $G = S_n$ and $H = S_a \wr S_a$, where $n=a^2$ and $a \geqs 3$. Then $b(G,H) = 3$.
\end{prop}

\begin{proof}
For $a \in \{3,4,5\}$ it is easy to verify the bound $b(G,H) \leqs 3$ using {\sc Magma}, so we will assume $a \geqs 6$. Set $k = \lfloor a/2 \rfloor$ and identify $G$ with ${\rm Sym}(X)$, where 
\[
X= \{1,\ldots,a\} \times \{1,\ldots,a\}.
\]
Let $\O$ be the set of partitions of $X$ into $a$ subsets of size $a$. As before, we need to identify three partitions in $\O$ whose pointwise stabiliser in $G$ is trivial.

To this end, define two partitions $\mathcal{B} = \{B_1, \ldots, B_a\}$ and $\mathcal{C} = \{C_1, \ldots, C_a\}$ in $\O$ by setting
\[
B_i = \{(x,y) \in X \, :\, x=i\}, \;\; C_i = \{(x,y) \in X \, :\, y=i \},\;\; i = 1,2, \ldots, a.
\]
We now define a third partition $\mathcal{D} = \{D_1, \ldots, D_a\}$ in $\O$. First we define $D_1, \ldots, D_{2k-2}$ by setting
\begin{align*}
D_{2i+1} & = \{(2i+2,2),(2i+1,2),(2i+2,4),(2i+1,4),\ldots, (2i+2,2i+2),(2i+1,2i+2), \\
& \hspace{7mm} (2i+1,2i+3),(2i+1,2i+4),\ldots,(2i+1,a)\} 
\end{align*}
and
\begin{align*}
D_{2i+2} & = \{(2i+1,1),(2i+2,1),(2i+1,3),(2i+2,3),\ldots, (2i+1,2i+1),(2i+2,2i+1), \\
& \hspace{7mm} (2i+2,2i+3),(2i+2,2i+4),\ldots,(2i+2,a)\},
\end{align*}
for $0 \leqs i \leqs k-2$. Then for $a$ even we define
\begin{align*}
D_{2k-1} & = \{(2k-1,1),(2k-1,2),(2k,2),(2k-1,4),(2k,4),\ldots, \\
& \hspace{7mm} (2k-1,2k-2),(2k,2k-2),(2k-1,2k)\} \\
D_{2k} & = \{(2k,1),(2k-1,3),(2k,3),(2k-1,5),(2k,5),\ldots, \\
& \hspace{7mm} (2k-1,2k-1),(2k,2k-1),(2k,2k)\}
\end{align*}
and for $a$ odd we set
\begin{align*}
D_{2k-1} & = \{(2k-1,1),(2k-1,3),(2k-1,4),(2k,4),(2k-1,6),(2k,6),\ldots, \\
& \hspace{7mm} (2k-1,2k),(2k,2k),(2k-1,2k+1)\} \\
D_{2k} & = \{(2k,1),(2k-1,2),(2k,2),(2k,3),(2k-1,5),(2k,5),(2k-1,7),(2k,7),\ldots, \\
& \hspace{7mm} (2k-1,2k-1),(2k,2k-1),(2k,2k+1)\} \\
D_{2k+1} & = \{(2k+1,1),(2k+1,2),\ldots,(2k+1,2k+1)\} = B_{2k+1} = B_a.
\end{align*}

Suppose $g \in \mbox{Sym}(X)$ stabilises $\mathcal{B}$, $\mathcal{C}$ and $\mathcal{D}$. We claim that $g=1$, which is sufficient to prove the proposition. To do this, we will show that $g$ fixes each block in $\mathcal{C}$ and $\mathcal{D}$.  Since $D_i \cap C_a = \{(i,a)\}$ for all $i$, if the previous condition holds then $g$ also fixes each block in $\mathcal{B}$ and the result then follows from the fact that $B_i \cap C_j = \{(i,j)\}$.

It will be convenient to adopt the following notation. For $i \in \{0,\ldots,a\}$ we define
\begin{align*}
c_r(i) & = |\{s \in \{1,\ldots,a\}\, :\, |C_r \cap D_s|=i\}| \\ 
d_r(i) & = |\{s \in \{1,\ldots,a\}\, :\, |C_s \cap D_r|=i\}|.
\end{align*}
Since $g$ is a bijection, observe that if $g(C_r)=C_t$ then $c_r(i)=c_t(i)$ for all $i$. Similarly, if $g(D_r)=D_t$ then $d_r(i)=d_t(i)$ for all $i$. In addition, notice that  $c_r(i) = d_r(i) = 0$ if $i \not \in \{0,1,2\}$. We now consider two cases, according to the parity of $a$.

\vs

\noindent \emph{Case 1. $a \geqs 7$ is odd.}

\vs

First we assume $a \geqs 7$ is odd. For $i \in \{0,1,2\}$, the values of $c_r(i)$ and $d_r(i)$ are recorded in Tables \ref{tab:aodd1} and \ref{tab:aodd2}, respectively. Let us highlight three immediate observations:
\begin{itemize}\addtolength{\itemsep}{0.2\baselineskip}
\item[{\rm (i)}] $c_r(1)=a$ if and only if $r=a$, so $g(C_a) = C_a$.
\item[{\rm (ii)}] $d_r(1)=a$ if and only if $r=a$, so $g(D_a) = D_a$.
\item[{\rm (iii)}] $c_r(1)=1$ if and only if $r=2$, so $g(C_2) = C_2$.
\end{itemize}

\begin{table}
\[
\begin{array}{lcccccc} \hline
& r=1 & r=2 & r=3 & r=4 & r \in \{2i+1,2i+2\},\,  2 \leqs i <k & r=a \\ \hline
c_r(0) & k-1 & k & k-2 & k-1 & k-i & 0 \\
c_r(1) & 3 & 1 & 5 & 3 & 2i+1 & a \\
c_r(2) & k-1 & k & k-2 & k-1 & k-i & 0 \\ \hline
\end{array}
\]
\caption{The values of $c_r(i)$ for $a \geqs 7$ odd, $i \in \{0,1,2\}$}
\label{tab:aodd1}
\end{table}

\begin{table}
\[
\begin{array}{lccc} \hline
& r \in \{2i+1,2i+2\},\, 0 \leqs i \leqs k-2 & r \in \{2k-1,2k\} & r=a \\ \hline
d_r(0) & i+1 & k-1 & 0 \\
d_r(1) & a-2i-2 & 3 & a \\
d_r(2) & i+1 & k-1 & 0 \\ \hline
\end{array}
\]
\caption{The values of $d_r(i)$ for $a \geqs 7$ odd, $i \in \{0,1,2\}$}
\label{tab:aodd2}
\end{table}

Suppose $i \in \{0,\ldots,k-3\}$. Since $d_r(1)=a-2i-2$ if and only if $r \in \{2i+1,2i+2\}$, it follows that $g$ stabilises each of the sets $\{D_{2i+1},D_{2i+2}\}$. In addition, since $|D_{2i+1} \cap C_2|=2$ and $|D_{2i+2} \cap C_2|=0$, the fact that $g$ fixes $C_2$ implies that $g$ also fixes $D_{2i+1}$ and $D_{2i+2}$ for all $i \in \{0,\ldots,k-3\}$. In particular, since $g(D_1) = D_1$ and $D_1 \cap C_r$ is empty if and only if $r=1$, it follows that $g(C_1) = C_1$.

Let us also observe that $d_r(1)=3$ if and only if $r \in \{2k-3,2k-2,2k-1,2k\}$, so the set $\{D_{2k-3},D_{2k-2},D_{2k-1},D_{2k}\}$ is stabilised by $g$. Now $|D_r \cap C_1|$ equals $0,2,1,1$ and $|D_r \cap C_2|$ equals $2,0,0,2$ for $r=2k-3$, $2k-2$, $2k-1$, $2k$, respectively. Since $C_1$, $C_2$ are fixed and $\{D_{2k-3},D_{2k-2},D_{2k-1},D_{2k}\}$ is stabilised, we deduce that $D_{2k-3}$, $D_{2k-2}$, $D_{2k-1}$ and $D_{2k}$ are fixed. We have now shown that $D_i$ is fixed for all $i \in \{1, \ldots, a\}$.

We know that $C_1,C_2$ and $C_a$ are fixed. Since $|D_3 \cap C_r|=0$ only if $r \leqs 3$, it follows that $C_3$ is fixed. Similarly, $|D_3 \cap C_r|=2$ only if $r \leqs 4$, so $C_4$ is fixed. In this way, if we assume that $C_1,\ldots,C_{2i}$ are fixed for some $i$ with $1 \leqs i < k$, then we can deduce that $C_{2i+1}$ and $C_{2i+2}$ are also fixed. Indeed, we have $|D_{2i+1} \cap C_r|=0$ only if $r \leqs 2i+1$, so $C_{2i+1}$ is fixed, and similarly $|D_{2i+1} \cap C_r|=2$ only if $r \leqs 2i+2$, so $C_{2i+2}$ is fixed.

We conclude that $g$ fixes every block in $\mathcal{C}$ and $\mathcal{D}$. As explained above, this forces $g=1$ as required.

\vs

\noindent \emph{Case 2. $a \geqs 6$ is even.}

\vs

To complete the proof of the proposition, we may assume that $a \geqs 6$ is even. 
For $i \in \{0,1,2\}$, we present the values of $c_r(i)$ and $d_r(i)$ in Tables \ref{tab:aeven1} and \ref{tab:aeven2} and we record the following sequence of deductions: 

\vspace{1mm}

\begin{itemize}\addtolength{\itemsep}{0.2\baselineskip}
\item[{\rm (i)}] $c_r(1)=a$ if and only if $r=a$, so $g(C_a) = C_a$.

\item[{\rm (ii)}] $c_r(1)=a-2$ if and only if $r=a-1$, so $g(C_{a-1}) = C_{a-1}$.

\item[{\rm (iii)}] $d_r(1)=a-2$ if and only if $r \in \{1,2\}$, so $g$ stabilises $\{D_1,D_2\}$.

\item[{\rm (iv)}] $|D_r \cap C_{a-1}|=0$ if and only if $r=a-1$ and $g(C_{a-1})= C_{a-1}$, so   
$g(D_{a-1}) = D_{a-1}$.

\item[{\rm (v)}] $|D_r \cap C_{a-1}|=2$ if and only if $r=a$ and $g(C_{a-1})= C_{a-1}$, so $g(D_a) = D_a$.

\item[{\rm (vi)}] $|D_{a-1} \cap C_r|=1$ if and only if $r \in \{1,a\}$, so $g(C_1) = C_1$ since $D_{a-1}$ and $C_a$ are fixed. 

\item[{\rm (vii)}] We have $|D_1 \cap C_1|=0$, $|D_2 \cap C_1|=2$ and $g(C_1)=C_1$. Since $g$ stabilises $\{D_1,D_2\}$, it follows that $g(D_i) = D_i$ for $i=1,2$. 

\item[{\rm (viii)}] $|D_1 \cap C_r|=2$ if and only if $r=2$, so $g(C_2) = C_2$ since $g(D_1) = D_1$. 
\end{itemize}

\vspace{1mm}

\begin{table}
\[
\begin{array}{lcccccc} \hline
& r=1 & r=2 & r \in \{2i+1,2i+2\},\, 1 \leqs i \leqs k-2 & r=2k-1 & r=2k \\ \hline
c_r(0) & k-1 & k & k-i & 1 & 0 \\ 
c_r(1) & 2 & 0 & 2i & 2k-2 & 2k \\ 
c_r(2) & k-1 & k & k-i & 1 & 0 \\ \hline
\end{array}
\]
\caption{The values of $c_r(i)$ for $a \geqs 6$ even, $i \in \{0,1,2\}$}
\label{tab:aeven1}
\end{table}

\begin{table}
\[
\begin{array}{lcc} \hline
& r \in \{2i+1,2i+2\},\, 0 \leqs i \leqs k-2 & r \in \{2k-1,2k\} \\ \hline
d_r(0) & i+1 & k-1  \\ 
d_r(1) & a-2i-2 & 2 \\ 
d_r(2) & i+1 & k-1 \\ \hline
\end{array}
\]
\caption{The values of $d_r(i)$ for $a \geqs 6$ even, $i \in \{0,1,2\}$}
\label{tab:aeven2}
\end{table}

Since $d_r(1)=a-2i-2$ if and only if $r \in \{2i+1,2i+2\}$, we deduce that $g$ stabilises each set $\{D_{2i+1},D_{2i+2}\}$ with $i=0,\ldots,k-3$. As noted above, $g$ fixes $D_{a-1}$ and $D_{a}$, so $\{D_{a-3},D_{a-2}\}$ is also stabilised. Now $|D_{2i+1} \cap C_1|=0$ and $|D_{2i+2} \cap C_1|=2$ for all $i=0,\ldots,k-2$, so the fact that $g(C_1) = C_1$ implies that $g$ fixes $D_{2i+1}$ and $D_{2i+2}$ for each $i$ in this range. We have now shown that $g$ fixes every block in $\mathcal{D}$.  

As explained above, we know that $C_1,C_2,C_{a-1}$ and $C_a$ are fixed by $g$. If we assume that $g$ fixes $C_1,\ldots,C_{2i}$ for some $i \leqs k-2$, then we can prove that $g$ also fixes $C_{2i+1}$ and $C_{2i+2}$. Indeed, $|D_{2i+1} \cap C_r|=0$ only if $r \leqs 2i+1$, so $C_{2i+1}$ is fixed. Similarly, $|D_{2i+1} \cap C_r|=2$ only if $r \leqs 2i+2$ and we deduce that $C_{2i+2}$ is fixed.

Therefore, by induction we see that $g$ fixes every block in $\mathcal{C}$ and $\mathcal{D}$, which implies that $g=1$, as explained above.
\end{proof}

By combining Propositions \ref{p:sym1}--\ref{p:plus0}, the proof of Theorem \ref{t:main2} is complete.

\begin{rem}\label{r:partan}
We can also determine the exact base size for the corresponding action of the alternating group. Set $G = A_n$ and $H = (S_b \wr S_a) \cap G$, where $n=ab$, $a \geqs b \geqs 2$ and $(a,b) \ne (2,2)$. Then as noted in \cite[Remark 5.3]{James}, we have $b(G,H) = 2$ if and only if $b \geqs 3$ and $a \geqs b+\e$, where
$\e=2$ if $b \geqs 5$, otherwise $\e=3$. One checks that $b(G,H) = 3$ if $(a,b) = (3,2)$ and so by combining the result in \cite{James} with Theorem \ref{t:main2}, we deduce that 
\[
b(G,H) = \left\{\begin{array}{ll}
2 & \mbox{if $b \geqs 3$ and $a \geqs b+\e$} \\
3 & \mbox{otherwise.}
\end{array}\right.
\]
\end{rem}

\section{Almost simple groups}\label{s:as}

In this section we prove Theorem \ref{t:main1}, so $G$ is an almost simple group and we will divide the proof into various parts, according to the structure of the socle $G_0$. Since $\a(G) \leqs \b(G)$,  it suffices to show that in the vast majority of cases, $G$ has a core-free maximal subgroup $H$ with $b(G,H) \leqs 3$. For the handful of exceptions with $G \cong S_6$ or $G_0 \cong {\rm U}_{4}(2)$, it is straightforward to check that $\a(G) \leqs 3$ unless $G \cong {\rm U}_{4}(2).2$, which gives the desired result. So our main aim throughout this section is to demonstrate the existence of a faithful primitive action of $G$ with base size $3$ (excluding the exceptions highlighted above).

There is an extensive literature on base sizes for primitive actions of almost simple groups. One of the main results in this area establishes a conjecture of Cameron from the 1990s and it is proved in the sequence of papers \cite{B07, BGS, BLSh, BOW}. This result states that if $G \leqs {\rm Sym}(\O)$ is an almost simple primitive group with point stabiliser $H$, then either $G$ is \emph{standard}, or $b(G,H) \leqs 7$ (with equality if and only if $G$ is the Mathieu group ${\rm M}_{24}$ in its natural action of degree $24$). Roughly speaking, the standard groups arise when $G_0 = A_n$ and $\O$ is a set of subsets or partitions of $[1,n]$, or $G_0$ is a classical group with natural module $V$ and $\O$ is a set of subspaces (or pairs of subspaces) of $V$. Stronger results have since been determined in a number of special cases. For instance, if $H$ is soluble then the precise base size of $G$ is computed in \cite{Bur20}, which shows that the bound $b(G,H) \leqs 5$ is best possible.

We will draw extensively on this earlier work and in several cases we will need to strengthen existing bounds on the base sizes of certain almost simple primitive groups. 

\subsection{Probabilistic methods}\label{ss:prob}

Before we begin the proof of Theorem \ref{t:main1}, we need to recall an important approach for deriving bounds on $b(G,H)$ in terms of fixed point ratio estimates. This probabilistic method was originally introduced by Liebeck and Shalev \cite{LSh99} and it plays an essential role in the proof of the aforementioned conjecture of Cameron on base sizes for almost simple groups.

Let $G \leqs {\rm Sym}(\O)$ be a finite transitive permutation group with point stabiliser $H$ and let $x_1, \ldots, x_k$ represent the conjugacy classes in $G$ of elements of prime order. For $x \in G$, let
\[
{\rm fpr}(x,\O) = \frac{|x^G \cap H|}{|x^G|}
\]
be the \emph{fixed point ratio} of $x$, which is simply the proportion of points in $\O$ fixed by $x$. For a positive integer $c$ we define
\begin{equation}\label{e:qhat}
\widehat{Q}(G,H,c) = \sum_{i=1}^{k}|x_i^G|\!\cdot\!{\rm fpr}(x_i,G/H)^c.
\end{equation}

Then as explained in the proof of \cite[Theorem 1.3]{LSh99}, the expression $\widehat{Q}(G,H,c)$ is an upper bound on the probability that a randomly chosen $c$-tuple of points in $\O$ does \emph{not} form a base for $G$. This yields the following result, which provides a useful method for bounding the base size $b(G,H)$.

\begin{lem}\label{l:basic}
If $\widehat{Q}(G,H,c) < 1$ then $b(G,H) \leqs c$.
\end{lem}

In order to estimate $\widehat{Q}(G,H,c)$, we will frequently apply \cite[Lemma 2.1]{B07}, which records the following basic observation.

\begin{lem}\label{l:bound}
Suppose $x_{1}, \ldots, x_{m}$ represent distinct $G$-classes such that $\sum_{i}{|x_{i}^{G}\cap H|}\leqs A$ and $|x_{i}^{G}|\geqs B$ for all $i$. Then 
\[
\sum_{i=1}^{m} |x_i^{G}| \cdot \left(\frac{|x_i^{G} \cap H|}{|x_i^{G}|}\right)^c \leqs B(A/B)^c
\]
for every positive integer $c$.
\end{lem}

\subsection{Alternating and sporadic groups}\label{ss:spor}

We begin the proof of Theorem \ref{t:main1} by handling the almost simple groups with socle an alternating or sporadic group. 

\begin{prop}\label{p:alt}
If $G$ is an almost simple group with socle $G_0 = A_n$, then $\b(G) \leqs 4$, with equality if and only if $G = S_6$.
\end{prop}

\begin{proof}
The groups with $n \leqs 14$ can be checked using {\sc Magma}, so we may assume $G = S_n$ or $A_n$ with $n \geqs 15$. If $G = A_n$, then the result follows from \cite[Theorem 1]{BGL}, so we can assume $G = S_n$. If $n= 2m$ is even, then Theorem \ref{t:main2} gives $b(G,H) = 3$ for $H = S_2 \wr S_m$.

To complete the proof, we may assume $G = S_n$ and $n \geqs 15$ is odd. Let $p \geqs 3$ be the smallest prime divisor of $n$. If $n=p^{k}$ for some $k \geqs 1$ then \cite[Theorem 1.1]{BGS} gives $b(G,H) = 2$ for $H = {\rm AGL}_{k}(p)$, otherwise Theorem \ref{t:main2} gives $b(G,H) \leqs 3$ for $H = S_p \wr S_{n/p}$.
\end{proof}

\begin{prop}\label{p:spor}
If $G$ is an almost simple sporadic group with socle $G_0$, then $\b(G) \leqs 3$, with equality if and only if $G_0 = {\rm M}_{22}$.
\end{prop}

\begin{proof}
This is an immediate corollary of the main theorem of \cite{BOW}.
\end{proof}

It is easy to show that $\a(S_6) = 3$, so by combining the two  propositions above we obtain the following corollary.

\begin{cor}\label{c:alpha1}
If $G$ is an almost simple group with socle an alternating or sporadic group, then $\a(G) \leqs 3$.
\end{cor}

\subsection{Exceptional groups}\label{ss:excep}

In this section we prove Theorem \ref{t:main1} in the case where $G_0$ is a simple exceptional group of Lie type. We will need the following lemma.

\begin{lem}\label{l:g2}
Let $G$ be an almost simple group with socle $G_0 = G_2(q)$, where $q = q_0^k$ for some prime $k$, and let $H$ be a subfield subgroup of type $G_2(q_0)$. Then $b(G,H) \leqs 3$.
\end{lem}

\begin{proof}
If $k$ is odd, then $b(G,H) = 2$ by \cite[Lemma 6.2]{BTh}, so for the remainder we may assume $k=2$. The case $q=4$ can be checked using {\sc Magma}, so we will assume $q \geqs 9$. By Lemma \ref{l:basic}, it suffices to show that $\widehat{Q}(G,H,3)<1$ (see \eqref{e:qhat}). As in the proof of \cite[Lemma 6.2]{BTh}, we proceed by estimating the contribution to $\widehat{Q}(G,H,3)$ from the various elements of prime order in $G$. Set $H_0 = H \cap G_0 = G_2(q_0)$ and write $q=p^f$ with $p$ a prime. We refer the reader to \cite{Chang,Eno} for detailed information on the conjugacy classes in $G$.

Let $x \in G_0$ be an element of prime order $r$ and first assume $r=2$. If $p \ne 2$ then $G_0$ and $H_0$ both have a unique conjugacy class of involutions and we obtain
\[
|x^G \cap H| = q^2(q^2+q+1) = u_1, \;\; |x^G| = q^4(q^4+q^2+1) = v_1.
\]
Similarly, if $p=2$ and $x$ is a long root element, then $|x^G \cap H| = q^3-1= u_2$ and $|x^G| = q^6-1=v_2$, whereas $|x^G \cap H| = q(q^3-1) = u_3$ and $|x^G| = q^2(q^6-1)=v_3$ if $x$ is a short root element. 

Next assume $r=p \geqs 3$. As above, the contribution to $\widehat{Q}(G,H,3)$ from long root elements is $v_2(u_2/v_2)^3$. Similarly, short root elements contribute $v_2(u_2/v_2)^3$ if $p=3$ and $v_3(u_3/v_3)^3$ if $p \geqs 5$. Since $v_3(u_3/v_3)^3< v_2(u_2/v_2)^3$, it follows that the combined contribution from long and short root elements is less than $2v_2(u_2/v_2)^3$ for all $p$. If $p=3$ and $x$ is in the class labelled $(\tilde{A}_1)_3$ in \cite[Table 22.2.6]{LieS} then $|x^G\cap H|<u_3$ and $|x^G|>v_3$. For all other unipotent elements, we have $|x^G|>\frac{1}{7}q^{10}=v_4$ and we note that $H_0$ contains precisely $u_4 = q^6$ unipotent elements in total, so Lemma \ref{l:bound} implies that the remaining unipotent contribution is less than $v_4(u_4/v_4)^3$.

Now assume $x \in G_0$ is semisimple and $r \geqs 3$. Let $\bar{G} = G_2(K)$ be the ambient simple algebraic group, where $K$ is the algebraic closure of $\mathbb{F}_q$. Then working with the standard Lie notation, we have $C_{\bar{G}}(x) = A_2$, $A_1T_1$ or $T_2$, where $T_i$ is an $i$-dimensional torus. If $C_{\bar{G}}(x) = A_2$ then $r=3$, $|x^G| \geqs q^3(q^3-1)=v_5$ and $H$ contains at most $u_5 = q^{3/2}(q^{3/2}+1)$ such elements. Similarly, if $C_{\bar{G}}(x) = A_1T_1$, then
\[
|x^G| \geqs \frac{|G_2(q)|}{|{\rm GU}_{2}(q)|} = q^5(q-1)(q^4+q^2+1) = v_6
\]
and we calculate that there are fewer than
\[
2q_0 \cdot \frac{|G_2(q_0)|}{|{\rm GL}_{2}(q_0)|} = 2q^3(q^{1/2}+1)(q^2+q+1) = u_6
\]
such elements in $H$. Finally, if $x$ is a regular semisimple element, then 
\[
|x^G| \geqs \frac{|G_2(q)|}{(q+1)^2} = q^6(q-1)(q^3-1)(q^2-q+1) = v_7
\]
and we note that $|H_0| = q^3(q-1)(q^3-1) = u_7$. 

To complete the analysis for $k=2$, we may assume $x \in G$ is a field automorphism of order $r$ (note that $G$ does not contain any graph automorphisms of prime order since $q = q_0^2$). If $r=2$ then we may assume $x$ centralises $H_0$, so 
\[
|x^G| = |G_0:H_0| = q^3(q^3+1)(q+1) = v_8
\]
and using \cite[Proposition 1.3]{LLS2} we deduce that
\[
|x^G \cap H| = 1+i_2(H_0) \leqs 2(q^{1/2}+1)q^{7/2} = u_8,
\]
where $i_2(H_0)$ denotes the number of involutions in $H_0$.
Now assume $r \geqs 3$, so $q_0 = q_1^r$ and $x$ acts as a field automorphism on $H_0$ (in particular, note that the condition $r \geqs 3$ implies that $q \geqs 2^6$). If $r=3$ then $|x^G|>\frac{1}{2}q^{28/3}=v_9$ and there are $2|G_2(q_0):G_2(q_0^{1/3})|<4q^{14/3} = u_9$ such elements in $H$. Similarly, if $r=5$ then $|x^G|>\frac{1}{2}q^{56/5}=v_{10}$ and $H$ contains $4|G_2(q_0):G_2(q_0^{1/5})|<8q^{28/5} = u_{10}$ such elements. Finally, if $r \geqs 7$ then $|x^G|>\frac{1}{2}q^{12} = v_{11}$ and we note that $|H|<2\log_2q.q^7 = u_{11}$.

By bringing the above bounds together, using Lemma \ref{l:bound}, we conclude that
\[
\widehat{Q}(G,H,3) < v_2(u_2/v_2)^3 + \sum_{i=1}^{8}v_i(u_i/v_i)^3 + \gamma\sum_{i=9}^{11}v_i(u_i/v_i)^3,
\]
where $\gamma=1$ if $q \geqs 2^6$, otherwise $\gamma=0$.
One checks that this upper bound is less than $1$ for $q \geqs 9$, whence $b(G,H) \leqs 3$ by Lemma \ref{l:basic}. 
\end{proof}

\begin{prop}\label{p:ex}
If $G$ is an almost simple group with socle $G_0$, an exceptional group of Lie type, then $\b(G) \leqs 3$.
\end{prop}

\begin{proof}
If $G$ has a maximal subgroup of the form $H = N_G(T)$, where $T$ is a maximal torus, then $b(G,H) = 2$ by \cite[Proposition 4.2]{BTh}. Therefore, by inspecting \cite[Table 5.2]{LSS}, we may assume that $G_0$ is one of the following:
\[
{}^2G_2(3)', G_2(q), F_4(q), {}^2E_6(2).
\]
The case $G_0 = {}^2G_2(3)'$ can be checked using {\sc Magma}: we get $\b(G_0) = 2$ and $\b(G_0.3) = 3$. Next assume $G_0 = F_4(q)$. Here $G$ has a maximal subgroup of type ${\rm L}_{3}(q)^2$ and once again \cite[Proposition 4.2]{BTh} gives $b(G,H) = 2$. The same conclusion holds if $G_0 = {}^2E_6(2)$ and $H$ is of type ${\rm L}_{3}(2)^3$.

Finally, let us assume $G_0 = G_2(q)$ and write $q=p^f$ with $p$ a prime. The cases with $q \leqs 5$ can be checked directly with the aid of {\sc Magma}, so we may assume $q \geqs 7$. If $G$ contains graph automorphisms, then $p=3$ and \cite[Table 5.2]{LSS} indicates that $G$ has a maximal subgroup $H$ of the form $N_G(T)$, where $T$ is a maximal torus. Once again, $b(G,H) = 2$ by \cite[Proposition 4.2]{BTh}. Next assume $f=1$, so $q=p$ is a prime and $H = {\rm U}_{3}(3){:}2$ is a maximal subgroup of $G$ (see \cite[Table 8.41]{BHR}, for example). Here the proof of \cite[Proposition 7.4]{BTh} gives $b(G,H) = 2$. Finally, if $q=q_0^k$ and $k$ is a prime, then Lemma \ref{l:g2} states that $b(G,H) \leqs 3$ with $H$ a subfield subgroup of type $G_2(q_0)$. The result now follows since $H$ is always maximal (see \cite{BHR}).  
\end{proof}

\subsection{Classical groups}\label{ss:class}

To complete the proof of Theorem \ref{t:main1}, we may assume $G_0$ is a finite simple classical group over $\mathbb{F}_q$, where $q=p^f$ with $p$ a prime. Due to the existence of exceptional isomorphisms among some of the low-dimensional classical groups (see \cite[Proposition 2.9.1]{KL}), we may assume $G_0$ is one of the following (here we adopt the notation for classical groups given in \cite{KL}):
\[
{\rm L}_{n}(q),\; n \geqs 2; \; {\rm U}_{n}(q), \; n \geqs 3; \; {\rm PSp}_{n}(q), \; n \geqs 4;
\]
\[
\O_n(q), \; \mbox{$nq$ odd, $n \geqs 7$;} \; {\rm P\O}_{n}^{\pm}(q), \; \mbox{$n \geqs 8$ even.}
\]
We will write $V$ to denote the natural module for $G_0$. 

The main result of this section is the following (note that there exist isomorphisms ${\rm U}_{4}(2) \cong {\rm PSp}_{4}(3)$ and ${\rm Sp}_{4}(2) \cong S_6$).

\begin{prop}\label{p:class}
Let $G$ be an almost simple classical group with socle $G_0$.
\begin{itemize}\addtolength{\itemsep}{0.2\baselineskip}
\item[{\rm (i)}]  We have $\a(G) \leqs 4$, with equality if and only if $G \cong {\rm U}_{4}(2).2$.
\item[{\rm (ii)}]  We have $\b(G) \leqs 4$, with equality if and only if $G \cong {\rm Sp}_{4}(2)$ or $G_0 \cong {\rm U}_{4}(2)$.
\end{itemize}
\end{prop}

Let $G$ be a finite almost simple classical group with socle $G_0$ and let $H$ be a core-free maximal subgroup of $G$ (that is, $H$ is maximal and $G = HG_0$). The main theorem on the subgroup structure of finite classical groups is due to Aschbacher. In \cite{asch}, Aschbacher proves that either $H$ is contained in one of eight \emph{geometric} subgroup collections, labelled $\C_1, \ldots, \C_8$, or $H$ is almost simple and the natural module $V$ is an absolutely irreducible module for a suitable covering group of the socle of $H$. The subgroups comprising each collection $\C_i$ are defined in terms of the geometry of $V$; they include the stabilisers of appropriate subspaces and direct sum decompositions of $V$, for example. If $\dim V \leqs 12$ then the complete list of maximal subgroups of $G$ (up to conjugacy) is determined in \cite{BHR}, while a similar result for the subgroups in the $\C_i$ collections is given in \cite{KL} for $\dim V>12$. We will repeatedly refer to both sources throughout this section. Following \cite{KL}, it will also be convenient to refer to the \emph{type} of a geometric subgroup of $G$, which provides an approximate description of the subgroup and the geometric structure it stabilises (see \cite[p.58]{KL}). 
 
As usual, we will view $G$ as a primitive permutation group on the set of cosets of $H$. If $H$ is soluble, then the base size $b(G,H)$ is determined in \cite{Bur20}. We will also repeatedly apply the following result from \cite{BGS0} concerning the field extension subgroups comprising the collection $\C_3$.

\begin{prop}\label{p:c3}
Let $G$ be an almost simple classical group with socle $G_0$ and assume $\dim V \geqs 6$, where $V$ is the natural module for $G_0$. If $H \in \C_3$ is a maximal field extension subgroup of $G$, then $b(G,H) \leqs 3$.
\end{prop}

\begin{proof}
This is a simplified version of \cite[Proposition 4.1]{BGS0}.
\end{proof}

The next result will also be useful.

\begin{prop}\label{p:eta}
Let $G$ be an almost simple classical group with socle $G_0$ and assume $\dim V \geqs 6$, where $V$ is the natural module for $G_0$. Suppose $H$ is a maximal subgroup of $G$ such that
\[
{\rm fpr}(x,G/H) < |x^G|^{-\frac{4}{9}}
\]
for all $x \in G$ of prime order. Then $b(G,H) \leqs 3$.
\end{prop}

\begin{proof}
We can repeat the proof of \cite[Proposition 6.3]{BGL}, which combines Lemma \ref{l:basic} with \cite[Proposition 2.2]{B07}.
\end{proof}

We are now ready to begin the proof of Theorem \ref{t:main1} for classical groups. We divide the analysis into several cases, according to the socle $G_0$. 

\subsubsection{Linear groups}\label{sss:linear}

In this section we assume $G_0 = {\rm L}_{n}(q)$ is a linear group, so $n \geqs 2$ and $(n,q) \ne (2,2), (2,3)$. 

\begin{prop}\label{p:linear}
If $G$ is an almost simple group with socle $G_0 = {\rm L}_{n}(q)$, then $\a(G) \leqs 3$. In addition, $\b(G) \leqs 4$, with equality if and only if $G = {\rm L}_{2}(9).2 \cong S_6$.
\end{prop}

\begin{proof}
Since ${\rm L}_{2}(9) \cong A_6$, we may assume that $(n,q) \ne (2,9)$ and our goal is to verify the bound $\b(G) \leqs 3$. If $n \geqs 6$ then by inspecting \cite{BHR,KL} we observe that $G$ always contains a maximal subgroup $H$ in the collection $\C_3$, so Proposition \ref{p:c3} gives $b(G,H) \leqs 3$ and the result follows.

If $G_0 = {\rm L}_{5}(q)$ then $G$ has a maximal subgroup $H$ of type ${\rm GL}_{1}(q^5)$, which is soluble. Here \cite[Theorem 2]{Bur20} gives $b(G,H) = 2$. Next assume $G_0 = {\rm L}_{4}(q)$. If $q \geqs 7$ then $G$ has a maximal $\C_2$-subgroup of type ${\rm GL}_{1}(q) \wr S_4$ and once again $b(G,H) = 2$ by \cite[Theorem 2]{Bur20}. The cases with $n=4$ and $q \leqs 5$ can be checked directly using {\sc Magma} \cite{magma}. Similar reasoning applies when $G_0 = {\rm L}_{3}(q)$: if $q \ne 4$ then $G$ has a maximal subgroup of type ${\rm GL}_{1}(q^3)$, while there is one of type ${\rm GU}_{3}(2)$ when $q=4$. In both cases, $H$ is soluble and \cite[Theorem 2]{Bur20} gives $b(G,H) \leqs 3$. 

Finally, let us assume $G_0 = {\rm L}_{2}(q)$. Here $G$ has a maximal subgroup $H$ of type ${\rm GL}_{1}(q) \wr S_2$ if $q$ is even and one of type ${\rm GL}_{1}(q^2)$ if $q \geqs 11$ is odd. Since $H$ is soluble in both cases, \cite[Theorem 2]{Bur20} implies that $b(G,H) \leqs 3$. The remaining cases with $q \in \{5,7\}$ can be handled using {\sc Magma}.
\end{proof}

\subsubsection{Unitary groups}\label{sss:unitary}

Next we assume $G_0 = {\rm U}_{n}(q)$, so $n \geqs 3$ and $(n,q) \ne (3,2)$. We will need the following technical result, which is an extension of \cite[Lemma 6.6]{BGL}.

\begin{lem}\label{l:uni}
Let $G$ be an almost simple group with socle $G_0 = {\rm U}_{n}(q)$, where $n=2^m$ and $m \geqs 3$. Let $H$ be a $\C_2$-subgroup of $G$ of type ${\rm GU}_{1}(q) \wr S_n$. Then
\[
{\rm fpr}(x,G/H) < |x^G|^{-\frac{4}{9}}
\]
for all $x \in G$ of prime order.
\end{lem}

\begin{proof}
Let $x \in G$ be an element of prime order $r$. The precise structure of $H$ is given in \cite[Proposition 4.2.9]{KL} and explicit bounds on ${\rm fpr}(x,G/H)$ are determined in the proofs of \cite[Propositions 2.5--2.7]{Bur3}. In particular, these bounds are used to establish the main theorem of \cite{Bur1} in this case, which states that  
\[
{\rm fpr}(x,G/H) < |x^G|^{-\frac{1}{2}+\frac{1}{n}}.
\]
In view of this bound, we may assume that $n \in \{8,16\}$. 

If $G = G_0$ then the desired result is \cite[Lemma 6.6]{BGL}, so we may assume $G \ne G_0$. A very similar argument handles all prime order elements $x \in {\rm PGU}_{n}(q)$, using essentially the same bounds on $|x^G \cap H|$ and $|x^G|$ given in the proof of \cite[Lemma 6.6]{BGL}. We omit the details. 

Finally, let us assume $x \in G \setminus {\rm PGU}_{n}(q)$, so $x$ is either an involutory graph automorphism or a field automorphism of odd prime order. In the latter case we have $q=q_0^r$ with $r \geqs 3$ and the proof of \cite[Proposition 2.7]{Bur3} gives
\[
|x^G \cap H| \leqs (q+1)^{n-1}n!,\;\; |x^G|>\frac{1}{2}\left(\frac{q}{q+1}\right)q^{(n^2-1)(1-r^{-1})-1}.
\]
It is straightforward to check that these bounds are sufficient. Now assume $x$ is an 
involutory graph automorphism. The case $G_0 = {\rm U}_{8}(2)$ can be handled using {\sc Magma}, so we may assume $(n,q) \ne (8,2)$. If $C_{G_0}(x)' = {\rm PSp}_{n}(q)$, then by appealing to the proof of \cite[Proposition 2.7]{Bur3} we get
\[
|x^G \cap H| \leqs (q+1)^{\frac{n}{2}-1}\cdot \frac{n!}{(n/2)!2^{n/2}},\;\; |x^G|>\frac{1}{2}\left(\frac{q}{q+1}\right)q^{\frac{1}{2}(n^2-n-4)}.
\]
Similarly, if $C_{G_0}(x)' \ne {\rm PSp}_{n}(q)$ then 
\[
|x^G \cap H| \leqs (q+1)^{n-1}\cdot (i_2(S_n)+1),\;\; |x^G|>\frac{1}{2}\left(\frac{q}{q+1}\right)q^{\frac{1}{2}(n^2+n-4)},
\]
where $i_2(S_n)$ denotes the number of involutions in $S_n$. Since $i_2(S_8) = 763$ and $i_2(S_{16}) = 46206735$, it is easy to verify the desired bound in both cases.
\end{proof}

\begin{prop}\label{p:unitary}
Let $G$ be an almost simple group with socle $G_0 = {\rm U}_{n}(q)$ and $n \geqs 3$. Then $\b(G) \leqs 4$, with equality if and only if $G_0 = {\rm U}_{4}(2)$.
\end{prop}

\begin{proof}
First assume $n \geqs 6$ is divisible by an odd prime $k$. The case $(n,q) = (6,2)$ can be handled using {\sc Magma}; in the remaining cases, by inspecting \cite{BHR,KL} we see that $G$ has a maximal subgroup of type ${\rm GU}_{n/k}(q^k)$, which is contained in the $\C_3$ collection. Therefore, Proposition \ref{p:c3} implies that $b(G,H) \leqs 3$.

Next assume $n=5$. If $q \geqs 3$ then $G$ has a soluble maximal subgroup of type ${\rm GU}_{1}(q^5)$ and \cite[Theorem 2]{Bur20} gives $b(G,H) = 2$. For $q=2$, a {\sc Magma} computation shows that $\b(G) = 2$. The case $n=4$ is similar. For $q \geqs 4$ we take a $\C_2$-subgroup $H$ of type ${\rm GU}_1(q) \wr S_4$ and apply \cite[Theorem 2]{Bur20}, whereas for $q =2,3$ we use {\sc Magma}. Similarly, if $n=3$ and $q \ne 5$ then \cite[Theorem 2]{Bur20} gives $b(G,H) \leqs 3$ for $H$ of type ${\rm GU}_{1}(q) \wr S_3$, while a {\sc Magma} calculation shows that $\b(G) \leqs 3$ when $q=5$.

Finally, let us assume $n = 2^m$ and $m \geqs 3$. Let $H$ be a $\C_2$-subgroup of type ${\rm GU}_{1}(q) \wr S_n$, which is always maximal in $G$ (see \cite{BHR, KL}). By combining Proposition \ref{p:eta} with Lemma \ref{l:uni}, we deduce that $b(G,H) \leqs 3$ and the result follows. 
\end{proof}

\begin{cor}\label{c:unitary}
Let $G$ be an almost simple group with socle $G_0 = {\rm U}_{n}(q)$. Then $\a(G) \leqs 4$, with equality if and only if $G = {\rm U}_{4}(2).2$.
\end{cor}

\begin{proof}
By the previous proposition, we may assume $G_0 = {\rm U}_{4}(2)$. Then with the aid of {\sc Magma}, it is easy to check that $G$ has maximal subgroups $H,K,L$ with $H \cap K \cap L = 1$ if and only if $G = G_0$. 
\end{proof}

\subsubsection{Symplectic groups}\label{sss:symplectic}

In this section we turn to the case $G_0 = {\rm PSp}_{n}(q)'$, where $n \geqs 4$. The case $n=4$ requires special attention and we will need the following lemmas. Note that ${\rm Sp}_{4}(2)' \cong A_6$ and ${\rm PSp}_{4}(3) \cong {\rm U}_{4}(2)$, so we may assume $q \geqs 4$ if $n=4$. 

\begin{lem}\label{l:sp41}
Let $G$ be an almost simple group with socle ${\rm Sp}_{4}(q)'$, where $q$ is even. Then $\b(G) \leqs 4$, with equality if and only if $G = {\rm Sp}_{4}(2) \cong S_6$.
\end{lem}

\begin{proof}
As noted above, we may assume $q \geqs 4$. If $G$ contains graph automorphisms, then \cite[Table 8.14]{BHR} indicates that $G$ has a maximal subgroup $H$ with $H \cap G_0 = (q+1)^2{:}D_8$. Here $H$ is soluble and \cite[Theorem 2]{Bur20} gives $b(G,H) = 2$. For the remainder we may assume $G \leqs {\rm \Gamma Sp}_{4}(q)$, where ${\rm \Gamma Sp}_{4}(q)$ denotes the subgroup of ${\rm Aut}(G_0)$ generated by the inner and field automorphisms.

Write $q = q_0^k$, where $k$ is a prime, and let $H$ be a maximal subfield subgroup of type ${\rm Sp}_{4}(q_0)$. We claim that $b(G,H) \leqs 3$. The cases with $q \leqs 2^5$ can be checked using {\sc Magma}, so we may assume $q \geqs 2^6$. In view of Lemma \ref{l:basic}, it suffices to show that $\widehat{Q}(G,H,3)<1$.

Let $x \in G$ be an element of prime order $r$. First assume $x$ is a unipotent involution, so $x$ is $G$-conjugate to $b_1$, $a_2$ or $c_2$ in the notation of Aschbacher and Seitz \cite{AS}. If $x$ is of type $b_1$ or $a_2$, then $|x^G \cap H| = q_0^4-1 \leqs q^2-1 = u_1$ and $|x^G| = q^4-1 = v_1$. Similarly, if $x$ is a $c_2$-involution then $|x^G \cap H| \leqs (q-1)(q^2-1) =u_2$ and $|x^G| = (q^2-1)(q^4-1) = v_2$.

Next assume $r$ is odd and $x$ is semisimple. We may assume that $x \in H_0$ (otherwise the contribution to $\widehat{Q}(G,H,3)$ from the elements in $x^G$ is zero). In particular, $r$ divides $q_0^4-1$. If $x$ is regular, then 
\[
|x^G| \geqs \frac{|{\rm Sp}_{4}(q)|}{(q+1)^2} = q^4(q-1)^2(q^2+1) = v_3
\]
and we note that $|H_0| \leqs q^2(q-1)(q^2-1) = u_3$. Now assume $x$ is non-regular, which implies that $r$ divides $q_0^2-1$. In particular, there are fewer than $2\log_2 q_0$ choices for $r$. Now 
\[
|x^G| \geqs \frac{|{\rm Sp}_{4}(q)|}{|{\rm GU}_{2}(q)|} = q^3(q^2+1)(q-1) = v_4
\]
and we calculate that there are fewer than
\begin{align*}
2\log_2q_0 \cdot \frac{1}{2}q_0 \cdot 2\left(\frac{|{\rm Sp}_{4}(q_0)|}{|{\rm GL}_{2}(q_0)|}\right) & = 2\log_2q_0.q_0^4(q_0^2+1)(q_0+1) \\
& \leqs \log_2q.q^2(q+1)(q^{1/2}+1) = u_4
\end{align*}
such elements in $H$. 

Finally, suppose $x \in G$ is a field automorphism of order $r$, so $q = q_1^r$. First assume $k \geqs 3$. If $r = 2$ then $x$ acts as a field automorphism on $H_0$, so 
\[
|x^G \cap H| = \frac{|{\rm Sp}_{4}(q_0)|}{|{\rm Sp}_{4}(q_0^{1/2})|} = q_0^2(q_0+1)(q_0^2+1) \leqs q^{2/3}(q^{1/3}+1)(q^{2/3}+1)
\]
and $|x^G| = q^2(q+1)(q^2+1) = v_5$. On the other hand, if $r \geqs 3$ then $|x^G|>\frac{1}{2}q^{20/3} = v_6$ and we note that $|H| \leqs \log_2q.|{\rm Sp}_{4}(q_0)|< \log_2q.q^{10/3}$. 

Now assume $k=2$. If $r=2$ then we may assume $x$ centralises $H_0$, so 
\[
|x^G \cap H| = i_2({\rm Sp}_{4}(q_0))+1 = q(q^2+q-1) = u_5
\]
and $|x^G| = v_5$ as above. For $r \geqs 3$ we have $|x^G|>v_6$ and we note that $H$ contains at most 
\[
\sum_{r \in \pi}(r-1)\cdot \frac{|{\rm Sp}_{4}(q_0)|}{|{\rm Sp}_{4}(q_0^{1/r})|} < 2\log_2q.q^{10/3} = u_6
\]
field automorphisms of odd prime order, where $\pi$ is the set of odd prime divisors of $\log_2q$.
 
In view of the above estimates and by applying Lemma \ref{l:bound}, we conclude that 
\[
\widehat{Q}(G,H,3) < v_1(u_1/v_1)^3 + \sum_{i=1}^{6}v_i(u_i/v_i)^3
\]
if $q \geqs 2^6$. It is routine to check that this upper bound is less than $1$.
\end{proof}

\begin{lem}\label{l:sp42}
Let $G$ be an almost simple group with socle $G_0={\rm PSp}_{4}(q)$, where $q$ is odd. Then $\b(G) \leqs 4$, with equality if and only if $q=3$. 
\end{lem}

\begin{proof}
As noted above, we have ${\rm PSp}_{4}(3) \cong {\rm U}_{4}(2)$, so we may assume that $q \geqs 5$. Write $q=p^f$ with $p$ a prime. Let $H$ be a $\C_2$-subgroup of type ${\rm GL}_{2}(q).2$, so $H$ is the stabiliser of a decomposition $V = U \oplus W$, where $U$ and $W$ are maximal totally isotropic subspaces of the natural module $V$. As recorded in \cite[Table 8.12]{BHR}, the condition $q \geqs 5$ implies that $H$ is a maximal subgroup of $G$. We claim that $b(G,H) \leqs 3$. 

We proceed as in the proof of \cite[Lemma 6.9]{BGL}, with the aim of constructing an explicit base of size $3$. To do this, let us first identify $G/H$ with the set $\O$ of pairs $\{U,W\}$, where $U$ and $W$ are totally isotropic $2$-spaces with $V = U \oplus W$. Fix a standard symplectic basis $\mathcal{B} = \{e_1,e_2,f_1,f_2\}$ for $V$ and consider $\a = \{U,W\}$ and $\b = \{U',W'\}$ in $\O$, where
\[
U = \la e_1, e_2 \ra,\;\; W = \la f_1, f_2\ra, \;\; U' = \la e_1, e_2+f_2 \ra,\;\; W' = \la e_1+f_2, e_2+f_1 \ra.
\]
Set $L = {\rm GSp}_{4}(q)$ and observe that
\[
L_{\a} = \left\{ \left(\begin{array}{cc} A & 0 \\ 0 & \l A^{-T} \end{array}\right),\; \left(\begin{array}{cc}  0 & \l A^{-T} \\ A & 0 \end{array}\right) \,:\, A \in {\rm GL}_{2}(q),\; \l \in \mathbb{F}_{q}^{\times} \right\}.
\]
An easy calculation shows that 
\[
L_{\a} \cap L_{\b} = Z(L) = \{\l I_4 \,: \, \l \in \mathbb{F}_{q}^{\times}\},
\]
so $b(G,H) = 2$ when $G \leqs {\rm PGSp}_{4}(q)$.

To complete the argument, we may assume $G$ contains field automorphisms. Write ${\rm Aut}(G_0) = {\rm PGSp}_{4}(q).\la \phi \ra$, where $\phi$ is a field automorphism of order $f$, which is defined with respect to the above standard basis $\mathcal{B}$. That is, the action of $\phi$ on $V$ is given by
\[
(a_1e_1+a_2e_2+a_3f_1+a_4f_2)^{\phi} = a_1^pe_1+a_2^pe_2+a_3^pf_1+a_4^pf_2.
\]
Note that the pointwise stabiliser of $\a$ and $\b$ in $L\la \phi \ra$ is precisely $Z(L)\la \phi \ra$.
Fix a generator $\mu$ for $\mathbb{F}_{q}^{\times}$ and set $\gamma = \{U'',W''\} \in \O$, where $U'' = \la e_1, \mu e_2+f_2 \ra$ and $W'' = W'$. Since $\gamma$ is not fixed by $\phi^i$ for any $i$ in the range $1 \leqs i < f$, we conclude that $\{\a,\b,\gamma\}$ is a base for $G$ and the result follows.
\end{proof}

\begin{prop}\label{p:symp}
Let $G$ be an almost simple group with socle $G_0 = {\rm PSp}_{n}(q)$ and $n \geqs 4$. Then $\b(G) \leqs 4$, with equality if and only if $G = {\rm Sp}_{4}(2) \cong S_6$ or $G_0 = {\rm PSp}_{4}(3) \cong {\rm U}_{4}(2)$.
\end{prop}

\begin{proof}
First assume $n \geqs 6$. If $n$ is divisible by an odd prime $k$, then let $H$ be a subgroup of type ${\rm Sp}_{n/k}(q^k)$, which is contained in the collection $\C_3$. Similarly, if $n = 2^m$ then we can take $H \in \C_3$ of type ${\rm Sp}_{n/2}(q^2)$. In both cases, $b(G,H) \leqs 3$ by Proposition \ref{p:c3}. Finally, the result for $n=4$ follows from Lemmas \ref{l:sp41} and \ref{l:sp42}.
\end{proof}

In view of Proposition \ref{p:symp} and Corollary \ref{c:unitary}, we obtain the following result.

\begin{cor}\label{c:symp}
Let $G$ be an almost simple group with socle $G_0 = {\rm PSp}_{n}(q)$ and $n \geqs 4$. Then $\a(G) \leqs 4$, with equality if and only if $G = {\rm PGSp}_{4}(3) \cong {\rm U}_{4}(2).2$.
\end{cor}

\subsubsection{Odd dimensional orthogonal groups}\label{sss:oddorth}

In this section we prove Theorem \ref{t:main1} for the groups with socle $G_0 = \O_n(q)$, where $nq$ is odd and $n \geqs 7$.

\begin{prop}\label{p:ortodd}
Let $G$ be an almost simple group with socle $G_0 = \O_{n}(q)$, where $nq$ is odd and $n \geqs 7$. Then $\b(G) \leqs 3$.
\end{prop}

\begin{proof}
Let $V$ be the natural module for $G_0$ and let $(\, ,\,)$ be the corresponding nondegenerate symmetric bilinear form on $V$. First assume $n = 4m+1$ and fix a standard basis
\[
\mathcal{B} = \{e_1,\dots,e_m,f_1\dots,f_m, e_1^{\ast},\dots,e_m^{\ast},f_1^{\ast},\dots,f_m^{\ast},x\}
\]
for $V$, where $(x,x)=1,$ $(e_i,f_i)=1$ and $(e_i^{\ast},f_i^{\ast})=1.$ We claim that $b(G,H) \leqs 3$, where $H$ is the stabiliser in $G$ of a $2m$-dimensional nondegenerate subspace of plus-type (recall that a nondegenerate $2m$-space is of \emph{plus-type} if it contains an $m$-dimensional totally singular subspace). We may identify $G/H$ with the set $\O$ of subspaces of $V$ of this form.

Following the proof of \cite[Theorem 6.11]{BGL}, set 
\begin{align*}
U & = \langle e_1,\dots,e_m,f_1,\dots,f_m \rangle \\
W & = \langle e_1+x, f_1+e_1^{\ast}, e_2+f_1^{\ast}, f_2+e_2^{\ast}, e_3+f_2^{\ast},\dots,e_m+f_{m-1}^{\ast}, f_m+e_{m}^{\ast} \rangle
\end{align*}
and note that $U,W \in \O$. The proof of \cite[Theorem 6.11]{BGL} shows that if $g \in {\rm SO}_{n}(q)$ fixes $U$ and $W$, then $g=1$ and thus $\{U,W\}$ is a base for $\O_n(q)$ and ${\rm SO}_{n}(q)$. Therefore, to complete the proof of the proposition for $n=4m+1$, we may assume $G$ contains field automorphisms.

Write $q=p^f$ with $f \geqs 2$ and let $\phi \in {\rm Aut}(G_0)$ be a standard field automorphism of order $f$, which is defined with respect to the above basis $\mathcal{B}$. In other words, if we take an arbitrary vector $v = a_1e_1 + \cdots + a_{n-1}f_m^{\ast} + a_{n}x \in V$, then
\[
v^{\phi} = a_1^pe_1 + \cdots + a_{n-1}^pf_m^{\ast} + a_{n}^px \in V.
\]
Fix a generator $\mu$ for $\mathbb{F}_q^{\times}$ and set 
\[
W' = \langle \mu e_1+x, f_1+e_1^{\ast}, e_2+f_1^{\ast}, f_2+e_2^{\ast}, e_3+f_2^{\ast},\dots,e_m+f_{m-1}^{\ast}, f_m+e_{m}^{\ast} \rangle.
\]
Then $W' \in \O$ and it is plain to see that $W'$ is not fixed by $\phi^i$ for any $1 \leqs i < f$. It follows that $\{U, W, W'\}$ is a base for ${\rm Aut}(G_0)$ and thus $b(G,H) \leqs 3$ as claimed.

A very similar argument applies when $n = 4m+3$. Here we take $\O$ to be the set of $(2m+1)$-dimensional nondegenerate subspaces $X$ of $V$ with the property that the orthogonal complement of $X$ in $V$ is a plus-type space. Fix a standard basis 
\[
\{ e_1,\dots,e_m,f_1\dots,f_m, e_1^{\ast},\dots,e_m^{\ast},f_1^{\ast},\dots,f_m^{\ast},e,f,x\}
\]
for $V$, where $(x,x)=1,$ $(e,f)=1,$ $(e_i,f_i)=1,$ and $(e_i^{\ast},f_i^{\ast})=1,$ and define $U,W \in \O$ as in the proof of \cite[Theorem 6.11]{BGL}. Then $\{U,W\}$ is a base for ${\rm SO}_{n}(q)$, and if we take 
\[
W' = \langle \mu e_1^{\ast}+x, e_1+f_1^{\ast}, f_1+e_2^{\ast},\dots, e_m+f_m^{\ast}, f_m+e 
\rangle
\]
where $\mathbb{F}_q^{\times} = \la \mu \ra$ as above, then $\{U,W,W'\}$ is a base for ${\rm Aut}(G_0)$. 
\end{proof}

\subsubsection{Even dimensional orthogonal groups}\label{sss:evenorth}

Here we complete the proof of Theorem \ref{t:main1} by handling the groups with socle $G_0 = {\rm P\O}_{n}^{\e}(q)$ with $n \geqs 8$ even. We begin by considering some special cases with $\e=+$.

\begin{lem}\label{l:ort1}
Let $G$ be an almost simple group with socle $G_0={\rm P\O}_{n}^{+}(q)$, where $n = 2^m$ and $m \geqs 3$. Then $\b(G) \leqs 3$. 
\end{lem}

\begin{proof}
First assume $m=3$. If $q \ne 3$ then \cite[Table 8.50]{BHR}, which is reproduced from \cite{K}, indicates that $G$ has a maximal $\C_2$-subgroup of type $O_{2}^{-}(q) \wr S_4$. In particular, $H$ is soluble and \cite[Theorem 2]{Bur20} gives $b(G,H) \leqs 3$. Similarly, if $q=3$ then we can take a maximal subgroup $H$ of type 
$O_{4}^{+}(3) \wr S_2$; once again, $H$ is soluble and we apply \cite[Theorem 2]{Bur20}.

Now assume $m \geqs 4$. By \cite{KL}, $G$ has a maximal subgroup of type 
$O_{n/2}^{+}(q) \wr S_2$. If $x \in G$ has prime order, then by applying the main theorem of \cite{Bur1} we get
\begin{equation}\label{e:bd}
{\rm fpr}(x,G/H) < |x^G|^{-\frac{1}{2}+\frac{1}{n}},
\end{equation}
which is less than $|x^G|^{-4/9}$ if $m \geqs 5$. Therefore, Proposition \ref{p:eta} implies that $b(G,H) \leqs 3$ if $m \geqs 5$. 

Finally, suppose $m=4$ and let $x_1, \ldots, x_k$ be representatives of the conjugacy classes in $G$ of elements of prime order. Following \cite{B07}, set
\begin{equation}\label{e:eta}
\eta_G(t) = \sum_{i=1}^{k}|x_i^G|^{-t}
\end{equation}
with $t \in \mathbb{R}$. As recorded in \cite[Remark 2.3]{B07}, we have $\eta_G(4/15)<1$ and by applying the bound in \eqref{e:bd} we deduce that
\[
\widehat{Q}(G,H,3) = \sum_{i=1}^{k}|x_i^G| \!\cdot\! {\rm fpr}(x_i,G/H)^3 < \eta_G(-1+3/2-3/16) = \eta_G(5/16) < \eta_G(4/15)<1.
\]
This implies that $b(G,H) \leqs 3$ and the proof of the lemma is complete.
\end{proof}

\begin{lem}\label{l:ort2}
Let $G$ be an almost simple group with socle $G_0={\rm P\O}_{10}^{+}(q)$ and let $H$ be a $\C_2$-subgroup of type $O_{2}^{+}(q) \wr S_5$. If $q \geqs 8$, then $b(G,H) \leqs 3$.
\end{lem}

\begin{proof}
Write $q=p^f$ with $p$ a prime and let $H$ be the stabiliser in $G$ of an orthogonal decomposition
\begin{equation}\label{e:sum}
V = V_1 \perp V_2 \perp V_3 \perp V_4 \perp V_5
\end{equation}
of the natural module $V$, where each $V_i$ is a nondegenerate $2$-space of plus-type. We will assume $q \geqs 8$, in which case $H$ is a maximal subgroup of $G$ by \cite[Table 8.66]{BHR}.
In view of Lemma \ref{l:basic}, it suffices to show that $\widehat{Q}(G,H,3)<1$. Let $x \in G$ be an element of prime order $r$.

First observe that $|H| \leqs \log_2q.2^5(q-1)^55! = u_1$, so Lemma \ref{l:bound} implies that the contribution to $\widehat{Q}(G,H,3)$ from the elements with $|x^G| > q^{14}=v_1$ is less than $v_1(u_1/v_1)^3$. 

For the remainder, we may assume $|x^G| \leqs q^{14}$. If $x$ is a field or graph-field automorphism, then \cite[Lemma 3.48]{Bur2} gives $|x^G|>\frac{1}{4}q^{45/2}$. Therefore the condition on $|x^G|$ implies that $x \in {\rm PGO}_{10}^{+}(q)$.  Without loss of generality, we may assume that $x \in H$, so $x$ stabilises the decomposition in \eqref{e:sum}.

Suppose $r=p$. If $p \geqs 3$ then $x$ acts as a $3$-cycle or a $5$-cycle on the set of summands in \eqref{e:sum}, so $p \in \{3,5\}$ and $x$ has Jordan form $[J_3^2,J_1^4]$ or $[J_5^2]$ on $V$ in the respective cases (here $J_i$ denotes a standard unipotent Jordan block of size $i$). In both cases, the order of $C_{G}(x)$ can be read off from \cite[Lemma 3.18]{Bur2} and it is easy to see that $|x^G|>q^{14}$. 

Now assume $r=p=2$. We adopt the standard notation for unipotent involutions from \cite{AS}. If $x$ is of type $b_1$, then $|x^G|>\frac{1}{2}q^9$ (see \cite[Proposition 3.22]{Bur2}) and we see that the elements in $x^G \cap H$ correspond to involutions in $O_{2}^{+}(q)^5$ of the form $(y,1,1,1,1)$, up to permutations. Therefore, $|x^G \cap H| \leqs 5(q-1)$. For all other unipotent involutions, one checks that $|x^G| > q^{14}$. Indeed, if $x$ is an $a_2$-type involution, then
\[
|x^G| = \frac{|O_{10}^{+}(q)|}{q^{13}|O_{6}^{+}(q)||{\rm Sp}_{2}(q)|} = (q^5-1)(q^4+1)(q^3+1)(q^2+1)>q^{14}
\]
(see \cite[Table 3.5.1]{BG_book}) and the bounds on $|x^G|$ presented in the proof of \cite[Proposition 3.22]{Bur2} are sufficient in the remaining cases.

Finally, let us assume $r \ne p$, so $x$ is semisimple. If $r=2$ then the condition on $|x^G|$ implies that $x$ acts as a reflection on $V$, with eigenvalues $[-I_{1},I_{9}]$. Here $|x^G|>\frac{1}{4}q^{9} = v_2$ and we note that $|x^G \cap H| \leqs 5(q-1)=u_2$. On the other hand, if $r \geqs 3$ then 
\[
|x^G| \geqs \frac{|O_{10}^{+}(q)|}{|O_{8}^{-}(q)|{\rm GU}_{1}(q)|} >\frac{1}{2}q^{16}
\]
and so none of these elements satisfy the bound $|x^G| \leqs q^{14}$.

By bringing the above estimates together, we conclude that
\[
\widehat{Q}(G,H,3) < v_1(u_1/v_1)^3 + v_2(u_2/v_2)^3 < 1
\]
and the result follows.
\end{proof}

\begin{lem}\label{l:ort3}
Let $G$ be an almost simple group with socle $G_0={\rm P\O}_{10}^{+}(q)$. Then $\b(G) \leqs 3$. 
\end{lem}

\begin{proof}
If $q \equiv 3 \imod{4}$, then \cite[Table 8.66]{BHR} indicates that $G$ has a maximal subgroup of type $O_5(q^2)$. Since $H$ is contained in the collection $\C_3$, we deduce that $b(G,H) \leqs 3$ 
by Proposition \ref{p:c3}. Similarly, if $q \geqs 8$ then $G$ has a maximal $\C_2$-subgroup $H$ of type $O_{2}^{+}(q) \wr S_5$ and Lemma \ref{l:ort2} gives $b(G,H) \leqs 3$. Therefore, to complete the proof of the lemma, we may assume that $q \in \{2,4,5\}$.

Suppose $q = 2$. If $G = \O_{10}^{+}(2)$ then we may apply \cite[Theorem 6.13]{BGL}, so we can assume $G = O_{10}^{+}(2)$. Here $G$ has a maximal $\C_2$-subgroup of type ${\rm GL}_{5}(2)$ and with the aid of {\sc Magma} it is easy to check that $b(G,H) \leqs 3$.

Next assume $q=4$. Using {\sc Magma}, we can construct $A = {\rm Aut}(G_0)$ as a permutation group of degree $487637$ and we can find an involution $y \in A$ such that $B = C_A(y) = O_{10}^{+}(2) \times 2$ is a maximal subfield subgroup of $A$. Then by random search, we find elements $x_1,x_2 \in G_0$ such that $B \cap B^{x_1} \cap B^{x_2} = 1$. This implies that $b(A,B) \leqs 3$ and it follows that $b(G,H) \leqs 3$ for $H \in \C_5$ of type $O_{10}^{+}(2)$.  

Finally, let us assume $q=5$. By inspecting \cite[Table 8.66]{BHR}, we observe that $G$ has a maximal $\C_2$-subgroup $H$ of type $O_2^{+}(5) \wr S_5$ or $O_1(5) \wr S_{10}$ (more precisely, the latter subgroups are maximal if $G \leqs {\rm PO}_{10}^{+}(5)$ and the former when $G \not\leqs {\rm PO}_{10}^{+}(5)$). Using the \texttt{ClassicalMaximals} function in {\sc Magma}, we can construct $H$ as a subgroup of ${\rm PGO}_{10}^{+}(5) = {\rm Aut}(G_0)$ and in both cases we find an element $x \in G_0$ such that $H \cap H^x = 1$. This implies that $b(G,H) = 2$ and the result follows.
\end{proof}

\begin{lem}\label{l:ort4}
Let $G$ be an almost simple group with socle $G_0={\rm P\O}_{n}^{+}(q)$, where $n = 2k$ and $k \geqs 7$ is a prime. Then $\b(G) \leqs 3$. 
\end{lem}

\begin{proof}
Suppose $G$ has a maximal subgroup $H$ that is not contained in the collection $\C_1$ (in other words, $H$ is a \emph{non-subspace} subgroup). In addition, let us assume for now that $H$ is not a $\C_2$-subgroup of type ${\rm GL}_{n/2}(q)$. For $t \in \mathbb{R}$, define $\eta_G(t)$ as in \eqref{e:eta} and recall that $\eta_G(4/15)<1$ (see \cite[Remark 2.3]{B07}). Then the main theorem of \cite{Bur1} implies that \eqref{e:bd} holds for all $x \in G$ of prime order, which in turn implies that 
\[
\widehat{Q}(G,H,3) < \eta_G(-1+3/2-3/14) = \eta_G(2/7) < \eta_G(4/15) < 1
\]
since $n \geqs 14$. Therefore $b(G,H) \leqs 3$.

If $q \geqs 7$, then by inspecting \cite{KL} we see that $G$ has a maximal $\C_2$-subgroup of type $O_{2}^{+}(q) \wr S_{n/2}$. Similarly, if $q=4$ then we can take a subfield subgroup of type $O_{n}^{+}(2)$, while for $q=3$ we can work with a $\C_3$-subgroup of type $O_{n/2}(q^2)$. Therefore, it remains to consider the cases $q \in \{2,5\}$.

Suppose $q=5$. By carefully inspecting the relevant tables in \cite[Chapter 3]{KL}, we deduce that $G$ has a maximal $\C_2$-subgroup $H$ of type $O_1(5) \wr S_n$ if $G \leqs {\rm PO}_{n}^{+}(5)$, and one of type $O_{2}^{+}(q) \wr S_{n/2}$ in the remaining cases. Therefore $b(G,H) \leqs 3$ and the result follows.

Finally, let us assume $q=2$. In view of \cite[Theorem 6.13]{BGL}, we may assume that $G = O_{n}^{+}(2)$. Let $H = {\rm GL}_{n/2}(2).2$ be a maximal $\C_2$-subgroup of $G$. Here the main theorem of \cite{Bur1} gives
\begin{equation}\label{e:c22}
{\rm fpr}(x,G/H) < |x^G|^{-\frac{1}{2}+\frac{1}{n}+\frac{1}{n-2}}
\end{equation}
for all $x \in G$ of prime order. As a consequence, if $k \geqs 17$ then 
\[
\widehat{Q}(G,H,3) < \eta_G(-1+3/2-3/34-3/32) = \eta_G(173/544) < \eta_G(4/15) < 1
\]
and thus $b(G,H) \leqs 3$. Therefore, we may assume that $k \in \{7,11,13\}$. 

Let $V$ be the natural module for $G$, let $U$ be a nondegenerate $(k-1)$-space of plus-type and let $H = O_{k-1}^{+}(2) \times O_{k+1}^{+}(2)$ be the stabiliser of $U$ in $G$. Then $H$ is a maximal subgroup of $G$. For $k \in \{7,11\}$, we can use {\sc Magma} to construct $H$ as a subgroup of $G$ (as matrix groups) and we can then use random search to find elements $x_1, x_2 \in G$ such that $H \cap H^{x_1} \cap H^{x_2} =1$. This implies that $b(G,H) \leqs 3$. 

Finally, suppose $G = O_{26}^{+}(2)$. Here we can consider the same approach and construct a maximal subgroup $H = O_{12}^{+}(2) \times O_{14}^{+}(2)$. However, we have been unable to compute the size of an intersection $H \cap H^{x_1} \cap H^{x_2}$, so we use a different method to handle this case. Define the zeta function $\eta_G(t)$ as above. Given the detailed information on conjugacy classes and centralisers in \cite[Section 3.5]{BG_book}, it is possible to calculate the size of each conjugacy class in $G$ containing elements of prime order and this allows us to  deduce that $\eta_G(1/16)<1$. Therefore, if we take a maximal $\C_2$-subgroup $H = {\rm GL}_{13}(2).2$, then \eqref{e:c22} holds for all $x \in G$ of prime order and thus
\[
\widehat{Q}(G,H,3) < \eta_G(-1+3/2-3/26-3/24) = \eta_G(27/104) < \eta_G(1/16) < 1.
\]
Therefore $b(G,H) \leqs 3$ and the proof of the lemma is complete.
\end{proof}

\begin{prop}\label{p:orteven}
Let $G$ be an almost simple group with socle $G_0 = {\rm P\O}_{n}^{\e}(q)$, where $n \geqs 8$ is even. Then $\b(G) \leqs 3$.
\end{prop}

\begin{proof}
Suppose $n$ is divisible by an odd prime $k$ with $n/k \geqs 4$. Then by inspecting \cite{BHR,KL} we see that $G$ has a maximal $\C_3$-subgroup $H$ of type $O_{n/k}^{\e}(q^k)$ and Proposition \ref{p:c3} gives $b(G,H) \leqs 3$. Therefore, we may assume that $n = 2^m$ or $2k$, where $m \geqs 3$ and $k \geqs 5$ is a prime. 

If $\e=-$ then $G$ has a maximal $\C_3$-subgroup $H$ of type $O_{n/2}^{-}(q^2)$ if $n=2^m$ and type ${\rm GU}_{n/2}(q)$ if $n=2k$ with $k$ prime. In both cases, we now apply Proposition \ref{p:c3} as before. Finally, for $\e=+$ we apply Lemmas \ref{l:ort1}, \ref{l:ort3} and \ref{l:ort4}.
\end{proof}

\vs

By combining Propositions \ref{p:alt}, \ref{p:spor}, \ref{p:ex} and \ref{p:class} with Corollary \ref{c:alpha1}, we conclude that the proof of Theorem \ref{t:main1} is complete.

\section{Soluble groups}\label{s:sol}

In this section we focus on the intersection number of finite soluble groups and we prove Theorem \ref{t:main3}, which can be viewed as a generalisation of \cite[Theorem 3.3]{Archer} on nilpotent groups.

We begin by recalling a result due to Wolf (see \cite[Theorem A]{wolf}).

\begin{thm}\label{wo}
Let $G$ be a finite supersoluble group and let $V$ be a faithful completely reducible $G$-module. Then there exist $x,y \in V$ such that $C_G(x) \cap C_G(y)=1.$
\end{thm}

For abelian groups, this can be strengthened as follows.

\begin{lem}\label{abel}
Let $G$ be a finite abelian group and let $V$ be a faithful completely reducible $G$-module. Then there exists $x\in V$ such that $C_G(x)=1.$	
\end{lem}

\begin{proof}
First decompose $V=V_1 \oplus \cdots \oplus V_n$ as a direct sum of irreducible $G$-modules and set $F_i=\End_GV_i.$ Since $G$ is abelian, we have $\dim_{F_i}V_i=1$ and $G/C_G(V_i)\leqs F_i^{\times}.$ In particular, $C_G(V_i)=C_G(v_i)$ for every $0\neq v_i\in V_i.$ So if we take $v=(v_1,\dots,v_n)\in V$ with $v_i\neq 0$, then   
\[
C_G(v)=\bigcap_{i=1}^{n}C_G(v_i)=\bigcap_{i=1}^{n}C_G(V_i)=C_G(V)=1
\]
and the result follows.
\end{proof}

We are now in a position to establish the first bound in Theorem \ref{t:main3}. Recall that if $G$ is a finite group, then $\lambda(G)$ denotes the chief length of $G$ (that is, the number of factors in a chief series for $G$).

\begin{thm}\label{t:thm3_1}
If $G$ is a finite soluble group, then $\alpha(G)\leqs \lambda(G).$
\end{thm}

\begin{proof}
We prove the theorem by induction on $\lambda(G).$ Without loss of generality, we may assume that $\frat(G)=1.$ Let $F$ be the Fitting subgroup of $G$, so 
\[
F=\prod_{i=1}^{n} V_i^{d_i}
\]
and $G = F \rtimes H$, where the $V_i$ are pairwise non-isomorphic irreducible  $H$-modules and the $d_i$ are positive integers. In particular, 
\[
\lambda(G)=\lambda(H)+\sum_{i=1}^{n}d_i.
\] 
We may assume that the indices $i$ are ordered in such a way that $V_i$ is central if and only if $i > r.$ Since $C_G(F)=F,$ it follows that 
\[
Z(G)=\prod_{i=r+1}^{n}V_i^{d_i}.
\]
	
There exist $s=\sum_{i>r}d_i$ maximal subgroups $M_1,\ldots, M_s$ of $G$ such that
\[
\bigcap_{j=1}^{s}M_j = \left(\prod_{i=1}^{r}V_i^{d_i}\right)\rtimes H\cong G/Z(G),
\] 
so it is not restrictive to assume $Z(G)=1$ and $r=n.$ Similarly, there exist $t=\sum_i (d_i-1)$ maximal subgroups $K_1,\ldots, K_t$ of $G$ with
\[
\bigcap_{j=1}^{t}K_j = \left(\prod_{i=1}^{n} V_i\right)\rtimes H.
\]	
Therefore, we may also assume that $d_i=1$ for $i=1, \ldots, n$, which means we have reduced the problem to the case where
\begin{equation}\label{e:eq}
G = \left(\prod_{i=1}^{n} V_i\right)\rtimes H \text{ and } C_F(H)=1.
\end{equation}
	
Let $L=\frat(H)$ and set $\ell=\lambda(H/L).$ By induction, there exist $\ell$ maximal subgroups $X_1,\dots,X_{\ell}$ of $H$ such that
	$L=X_1\cap\dots\cap X_{\ell}.$ But then $Y_1=FX_1,\dots,Y_\ell=FX_{\ell}$ are maximal subgroups of $G$ with 
	\[
	FL=Y_1\cap\dots\cap Y_{\ell}.
	\]
		
	Assume $L\neq 1$ and set $L_1=L.$
	Since $C_H(F)=1,$ there exists $i_1\in \{1,\dots,n\}$ such that $L_2=C_{L_1}(V_{i_1})<L_1.$ 
	Notice that $L_1C_H(V_{i_1})/C_H(V_{i_1})\cong L_1/L_2$ is a normal subgroup of $H/C_H(V_{i_1}),$ which acts irreducibly on $V_{i_1},$ so Clifford Theory implies that $V_{i_1}$ is a faithful completely reducible $(L_1C_H(V_{i_1})/C_H(V_{i_1}))$-module. If $L_1/L_2$ is abelian, then by Lemma \ref{abel} there exists $v_{1,1}$ in $V_{i_1}$ such that $C_{L_1}(v_{1,1})=L_2.$ In any case, $L_1$ is nilpotent and thus  Theorem \ref{wo} implies that there exist two elements $v_{1,1}$ and $v_{1,2}$ in $V_{i_1}$ such that $L_2=C_{L_1}(v_{1,1})\cap C_{L_1}(v_{1,2}).$ Now set $\eta_1=1$ if $L_1/L_2$ is abelian, $\eta_1=2$ otherwise. Then 
		$$L_{1,0}=\left(\prod_{i\neq i_1} V_i\right)\rtimes H \; \mbox{ and } \; 
	L_{1,j}=\left(\prod_{i\neq i_1} V_i\right)\rtimes H^{v_{1,j}} \mbox{ for $1\leqs j \leqs \eta_1$}
	$$
	are maximal subgroups of $G$ and we observe that 
	$$\left(\bigcap_{i=1}^{\ell} Y_i\right)\cap \left(\bigcap_{i=0}^{\eta_1}L_{1,i}\right)=\left(\prod_{i\neq i_1} V_i\right)\rtimes L_2.
	$$
	
	If $L_2\neq 1,$ then there exists $i_2\neq i_1$ such that
	$L_3=C_{L_2}(V_{i_2})<L_2.$   Define $\eta_2=1$ if $L_2/L_3$ is abelian, $\eta_2=2$ otherwise. As before, we can find $v_{2,1}, v_{2,\eta_2}$ in $V_{i_2}$ such that $L_3=C_{L_2}(v_{2,1})\cap C_{L_2}(v_{2,\eta_2}).$ Then 
	$$L_{2,0}=\left(\prod_{i\neq i_2} V_i\right)\rtimes H \; \mbox{ and } \; 
	L_{2,j}=\left(\prod_{i\neq i_2} V_i\right)\rtimes H^{v_{2,j}} \mbox{ for $1\leqs j \leqs \eta_2$}
	$$
	are maximal subgroups of $G$ and
	$$\left(\bigcap_{i=1}^{\ell}Y_i\right)\cap \left(\bigcap_{i=0}^{\eta_1}L_{1,i}\right)\cap \left(\bigcap_{i=0}^{\eta_2}L_{2,i}\right)
	=\left(\prod_{i \ne i_1,i_2} V_i\right)\rtimes L_3.
	$$
	
	We repeat this procedure until we have $L_{k+1}=C_{L_k}(V_{i_k})=1.$ In this way, we construct a subset $J=\{i_1,\dots,i_k\}$ of $\{1,\dots,n\}$ and $k+\sum_{1\leqs i\leqs k}\eta_i$ maximal subgroups
	$$L_{1,0}, \, L_{1,1}, \, L_{1,\eta_1},\dots,L_{k,0}, \, L_{k,1}, \, L_{k,\eta_k}$$ of $G$ such that
	$$	\left(\bigcap_{i=1}^{\ell}Y_i\right)\cap \left(\bigcap_{i=0}^{\eta_1}L_{1,i}\right)\cap\dots\cap \left(\bigcap_{i=0}^{\eta_k}L_{k,i}\right)=\prod_{j\notin J}V_j.
	$$
	Finally, for each $j \in \{1, \ldots, n\} \setminus J$, let $R_j=\left(\prod_{i\neq j} V_i\right)\rtimes H.$ Then $R_j$ is a maximal subgroup of $G$ and we have
	$$\left(\bigcap_{i=1}^{\ell}Y_i\right)\cap \left(\bigcap_{i=0}^{\eta_1}L_{1,i}\right)\cap\dots\cap \left(\bigcap_{i=0}^{\eta_k}L_{k,i}\right)\cap 
	\left(\bigcap_{j\notin J}R_j\right)=1.
	$$
	Therefore, $\alpha(G)\leqs \ell+n+\sum_{1\leqs i \leqs k}\eta_i\leqs \lambda(G).$
\end{proof}

Let $G$ be a finite group and let $A=H/K$ be a chief factor of $G$. Recall that a subgroup $L \leqs G$ is a complement to $A$ in $G$ if $G = LH$ and $L \cap H=K$. We say that $A$ is a \emph{Frattini chief factor} if it is contained in the Frattini subgroup of $G/K$; equivalently, $A$ is abelian and there is no complement to $A$ in $G$. Let $\delta(G)$ denote the number of non-Frattini factors in a chief series for $G.$ 

The next result completes the proof of Theorem \ref{t:main3}.

\begin{thm}\label{t:thm3_2}
Let $G$ be a finite soluble group and assume the derived subgroup of $G$ is nilpotent. Then $\alpha(G) \leqs \delta(G).$
\end{thm}

\begin{proof}
	We may assume $\frat(G)=1.$ Let $F$ be the Fitting subgroup of $G$ and write 
	$$F=\prod_{i=1}^{n} V_i^{d_i}$$ 
	so that $G = F \rtimes H$ and $H$ is abelian. Here the $V_i$ are pairwise non-isomorphic irreducible $H$-modules and we have
	\[
	\delta(G)=\delta(H)+\sum_{i=1}^{n}d_i.
	\]
As in the proof of the previous theorem, we may assume that the indices $i$ are ordered so that $V_i$ is central if and only if $i > r.$ Since $C_G(F)=F,$ it follows that 
\[
Z(G)=\prod_{i=r+1}^{n} V_i^{d_i}.
\] 
Then by arguing as in the proof of the previous result, we may assume without loss of generality that $Z(G) = 1$, $r=n$ and \eqref{e:eq} holds.

	Let $H_1=H$ and $T_1=\soc(H_1).$ Since $C_H(F)=1,$ there exists $i_1\in \{1,\dots,n\}$ such that $C_{T_1}(V_{i_1})<T_1.$
	By Lemma \ref{abel}, there exists $v_1 \in V_{i_1}$ such that $C_H(v_1)=C_H(V_{i_1}).$
	Then 
		$$L_{1,1}=\left(\prod_{i\neq i_1} V_i\right)\rtimes H \; \mbox{ and } \;  L_{1,2}=\left(\prod_{i\neq i_1} V_i\right)\rtimes H^{v_1}
	$$
	are maximal subgroups of $G$. In addition, for $H_2=C_H(V_{i_1})$ we have 
	$$L_{1,1}\cap L_{1,2}=\left(\prod_{i\neq i_1} V_i\right)\rtimes H_2.
	$$
	
	Now let $T_2=\soc(H_2)$, so $T_2<T_1.$
	If $T_2\neq 1,$ then there exists $i_2\neq i_1$ such that
	$C_{T_2}(V_{i_2})<T_2.$ As before, there exists $v_2$ in $V_{i_2}$ such that $C_H(v_2)=C_H(V_{i_2})$ and we may consider the maximal subgroups
	$$L_{2,1}=\left(\prod_{i\neq i_2} V_i\right)\rtimes H \; \mbox{ and } \; 
	L_{2,2}=\left(\prod_{i\neq i_2} V_i\right)\rtimes H^{v_2}.
	$$
	If we set $H_3=C_{H_2}(V_{i_2})$, then 
	$$L_{1,1}\cap L_{1,2}\cap L_{2,1}\cap L_{2,2}=\left(\prod_{i \ne i_1,i_2} V_i\right)\rtimes H_3.
	$$
	
	We now repeat this procedure until we have $\soc(H_{k+1})=1.$ This implies that $H_{k+1}=1$, so in this way we construct a subset $J=\{i_1,\dots,i_k\}$ of $\{1,\dots,n\}$ and $2k$ maximal subgroups
	$L_{1,1}$, $L_{1,2},\dots,L_{k,1}$, $L_{k,2}$ of $G$ such that
	$$\left(\bigcap_{i=1}^{k}L_{i,1}\right)\cap \left(\bigcap_{i=1}^{k}L_{i,2}\right)=\prod_{j\notin J}V_j.
	$$
	Finally, for each $j \in \{1, \ldots, n\} \setminus J$, let $R_j=\left(\prod_{i\neq j} V_i\right)\rtimes H.$ Then each $R_j$ is a maximal subgroup of $G$ and 
	$$\left(\bigcap_{i=1}^{k}L_{i,1}\right)\cap \left(\bigcap_{i=1}^{k}L_{i,2}\right)\cap 
	\left(\bigcap_{j\notin J}R_j\right)=1,
	$$
	which implies that $\alpha(G)\leqs n+k.$ Now $|T_1|=|\soc(H)|$ is the product of precisely $\delta(H)$ (not necessarily distinct) prime numbers and thus $k \leqs \delta(H)$ since $T_{i+1}<T_{i}$ for $1\leqs i < k.$ The result follows.  
\end{proof}

\section{A general bound}\label{s:fin}

In this final section we prove Theorem \ref{t:main4}, which provides a general upper bound on the intersection number of an arbitrary finite group. We begin by briefly recalling the theory of crowns, which plays a key role in the proof. 

\subsection{Crowns}\label{ss:crowns}

Given groups $G$ and $A$, we say that $A$ is a \emph{$G$-group} if $G$ acts on $A$ via automorphisms. In addition, $G$ is \emph{irreducible} if it does not stabilise any nontrivial proper subgroups of $A$. Two $G$-groups $A$ and $B$ are \emph{$G$-isomorphic}, denoted by $A \cong_G B$, if there exists a group isomorphism $\varphi : A \to B$ such that $\varphi(g(a)) = g(\varphi(a))$ for all $a \in A$, $g \in G$. 

\begin{defn}\label{d:equiv}
Following \cite{paz}, we say that two irreducible $G$-groups $A$ and $B$  are \emph{$G$-equivalent}, denoted $A \sim_G B$, if there is an isomorphism $\Phi: A\rtimes G \rightarrow B\rtimes G$ such that the following diagram commutes:

\vspace{1mm}

\[
\begin{CD}
1@>>>A@>>>A\rtimes G@>>>G@>>>1\\
@. @VV{\varphi}V @VV{\Phi}V @|\\
1@>>>B@>>>B\rtimes G@>>>G@>>>1
\end{CD}
\]

\vspace{1mm}
\end{defn}

Observe that two $G$\nobreakdash-isomorphic $G$\nobreakdash-groups are $G$\nobreakdash-equivalent, and the converse holds if $A$ and $B$ are abelian. By \cite[Proposition 1.4]{paz}, two chief factors $A$ and $B$ of $G$ are $G$-equivalent if and only if  
\begin{itemize}\addtolength{\itemsep}{0.2\baselineskip}
\item[{\rm (a)}] they are $G$-isomorphic; or 
\item[{\rm (b)}] there exists a maximal subgroup $H$ of $G$ such that $G/H_G$ has two minimal normal subgroups $N_1$ and $N_2$ which are $G$-isomorphic to $A$ and $B$, respectively (here $H_G$ denotes the core of $H$ in $G$).
\end{itemize}

Let $L$ be a monolithic primitive group, so $L$ has a unique minimal
normal subgroup $A$, which is not contained in the Frattini subgroup ${\rm Frat}(L)$. Let $k$ be a positive integer and let $L^k$ be the direct product of $k$ copies of $L$. The \emph{crown-based power of $L$ of size $k$} is the subgroup $L_k$ of $L^k$ defined by
\[
L_k=\{(l_1, \ldots , l_k) \in L^k  \,: \, l_1 \equiv \cdots \equiv l_k \, {\rm mod}\, A \}.
\]
Equivalently, $L_k=A^k{\rm diag}(L^k)$, where ${\rm diag}(L^k) = \{(l, \ldots, l) \, : \, l \in L\} \leqs L^k$. Note that the minimal normal subgroups of $L_k$ are all $L_k$-equivalent.

\begin{rem}\label{r:crown}
In the setting of finite soluble groups, the notion of a crown-based power was first introduced by Gasch\"{u}tz in \cite{Gas}, and it was subsequently extended to all finite groups by Dalla Volta and Lucchini \cite{DVL}. For a more detailed exposition of the theory, we refer the reader to \cite[Section 1.3]{BBE}.
\end{rem}

Let $G$ be a finite group and recall that a chief factor $A=H/K$ of $G$ is \emph{Frattini} if $A \leqs {\rm Frat}(G/K)$. Let $\delta_G(A)$ be the number of non-Frattini chief factors in a chief series for $G$ which are $G$-equivalent to $A$ (this does not depend on the choice of chief series). Let 
$$L_{A}=
\begin{cases}
A\rtimes (G/C_G(A)) & \text{ if $A$ is abelian}, \\
G/C_G(A)& \text{ otherwise}
\end{cases}
$$
be the monolithic primitive group associated to $A$. If $A$ is non-Frattini, then $L_A$ is a homomorphic image of $G$; more precisely, there exists
a normal subgroup $N$ of $G$ such that $G/N \cong L_A$ and $\soc(G/N)\sim_GA$. 

Let $R_G(A)$ be the intersection of all the normal subgroups $N$ of $G$ with the property that $G/N \cong L_A$ and $\soc(G/N)\sim_G A$. Then $G/R_G(A)$ is isomorphic to the crown-based  power $(L_A)_{\delta_G(A)}$. The socle $I_G(A)/R_G(A)$ of $G/R_G(A)$ is called the \emph{$A$-crown} of $G$ and it is a direct product of $\delta_G(A)$ minimal normal subgroups that are $G$-equivalent to $A$. 

\begin{prop}\label{general}
	Let $G$ be a finite group and let $\mathcal B$ be the set of non-Frattini chief factors of $G$ that are $G$-equivalent to some minimal normal subgroup of $G/\frat(G).$ Then
	\[
	\frat(G)=\bigcap_{A\in \mathcal B}R_G(A).
	\]
\end{prop}

\begin{proof}
	Let $A=H/K$ be a non-Frattini chief factor of $G$ and set $R=R_G(A)$. Let $L \geqs K$ be a maximal subgroup of $G$ not containing $H.$ Then the core $L_G$ is a normal subgroup of $G$ centralising $A$ and containing $R$, whence $R \leqs C_G(A).$

	Next let $F^*(G)$ be the generalised Fitting subgroup of $G$. Without loss of generality, we may assume that $\frat(G)=1,$ so $F^*(G)=\soc(G)$ and  
	\[
	M:=\bigcap_{A\in \mathcal B}R_G(A)\leqs \bigcap_{A\in \mathcal B}C_G(A)=C_G(F^*(G))\leqs F^*(G).
	\]
	 Suppose $M\neq 1$ and let $N$ be a minimal normal subgroup of $G$ contained in $M.$ We have $N\sim_G A$ for some $A\in \mathcal B$, so  \cite[Lemma 10]{crprof} implies that $N\not\leqs R_G(A).$ However $N\leqs M\leqs R_G(A)$ by definition, so we have reached a contradiction and thus $M=1$. The result follows.
\end{proof}

\subsection{The intersection number}

We now turn to the proof of Theorem \ref{t:main4}. We need some preliminary lemmas.

Given a finite group $H$, a faithful irreducible $H$-module $A$ and a positive integer $t$, consider the semidirect product $G=A^t\rtimes H.$ Let 
\[
\Der(H,A)=\{\delta: H \to A \,:\,  (h_1h_2)^\delta=(h_1^\delta)^{h_2}h_2^\delta \mbox{ for every $h_1, h_2 \in H$} \}
\]
be the set of derivations from $H$ to $A$. Recall that the map $\delta\mapsto \{hh^\delta \,:\, h \in H\}$ induces a bijection from $\Der(H,A)$ to the set of complements of $A$ in $A\rtimes H.$ For each $\delta \in \Der(H,A)$, define $C_\delta=\{h\in H \,:\, h^\delta=0\}$ and set 
\[
\Lambda(H,A) = \mathcal{M} \cup \{C_\delta \,:\, \delta\in \Der(H,A)\},
\]
where $\mathcal{M}$ is the set of maximal subgroups of $H$. Finally, let $\sigma(H,A)$ be the minimal cardinality of a family of subgroups in $\Lambda(H,A)$ with trivial intersection.
		
		\begin{lem}\label{l:ab0}
Let $G=A^t\rtimes H$ be a finite group as above. Then 
\begin{itemize}\addtolength{\itemsep}{0.2\baselineskip}
\item[{\rm (i)}] $\alpha(G)=t+\sigma(H,A)$; and
\item[{\rm (ii)}] $\alpha(G) \leqs	t + b(H,A),$ where $b(H,A)$ is the base size of $H$ on A. 
\end{itemize}
As a consequence, $\alpha(G)\leqs t + \dim_{\End_H(A)}A$, with $\alpha(G) \leqs  t + 3$ 
if $H$ is soluble. 
\end{lem} 
		
		\begin{proof}
			Set $\alpha=\alpha(G)$ and $\sigma=\sigma(H,A).$ First we prove 
			that $\alpha \leqs t+\sigma.$ We may write $G= (A_1\times \dots\times A_t)\rtimes H,$
			where  $A_i\cong_H A$ for every $i.$ 
			For $i = 1, \ldots, t$, set 
			$$M_i=\left(\prod_{j\neq i} A_j\right)\rtimes H.$$ Since
			$M_1\cap \dots \cap M_{t-1}\cong A\rtimes H,$
			it suffices to show that $\alpha(A\rtimes H)\leqs 1+\sigma.$
			
			By definition of $\s$, there exist $K_j\in \Lambda(H,A)$ such that 
			$K_1\cap \cdots \cap K_{\sigma}=1$. By relabelling, if necessary, we may assume that there exists $r\leqs \sigma$ such that $K_j=C_{\delta_j}$ for some $\delta_j\in \Der(H,A)$ if $j\leqs r,$ while $K_j$ is a maximal subgroup of $H$ if $j>r$. Notice that if $j\leqs r,$  then $Y_j := \{hh^{\delta_j} \,:\, h \in H\}$ is a complement of $A$ in $A\rtimes H$. On the other hand, if $j>r$ then $Y_j :=AK_j$ is a maximal subgroup of $A\rtimes H.$ It follows that 
			\[
			H\cap Y_1 \cap \cdots \cap Y_\sigma=K_1\cap \cdots \cap K_\sigma=1
			\]
		and thus $\alpha(A\rtimes H)\leqs 1+\sigma$ as required.
			
		Next we show that $\a \geqs t+\s$. Let $M_1,\ldots,M_\alpha$ be maximal subgroups of $G$ such that $M_1\cap \cdots \cap M_\alpha=1.$ For $1\leqs j\leqs \alpha,$ let $B_j$ be the normal subgroup $A^t\cap M_j$ of $G$. Note that if $B_j \neq A^t$ then $A^t/B_j \cong_G A$. Therefore, $\alpha\geqs t$ and we may assume that $M_1,\dots,M_t$ are maximal supplements of $A^t$ in $G$, with 
		\[
		M_1\cap \dots \cap M_t\cap A^t = B_1 \cap \dots \cap B_t=1.
		\]
		(Recall that a subgroup $K$ of $G$ is a \emph{supplement} of $A^t$ if $G = A^tK$.)
		 In particular, note that $\{B_1,\ldots,B_t\}$ is \emph{irredundant} in the sense that the intersection of any proper subset is nontrivial. 
		 
		 We claim that $M_1\cap \dots \cap M_t$ is a complement of $A^t$ in $G$. To do this, we will use induction on $i$ to show that $X_i:=M_1\cap \dots \cap M_i$ is a supplement of $A^t$ in $G$ for all $1 \leqs i\leqs t.$ Assume this for some $i<t$ and let $C_i=B_1\cap \dots \cap B_i.$ Since $B_{i+1}$ is a maximal $G$-subgroup of $A^t$ and $\{B_1,\ldots,B_t\}$ is irredundant, it follows that $C_iB_{i+1}=A^t.$ Now let $h\in H$. Since $A^tX_i=G=A^tM_{i+1},$ there exist $a_1, a_2\in A^t$ such that
			$a_1h\in X_i$ and $a_2h\in M_{i+1}.$ Moreover there exist $c\in C_i$ and $b\in B_{i+1}$ such that $a_1a_2^{-1}=bc^{-1}.$ Hence
			$(ca_1)h=(ba_2)h\in X_i\cap M_{i+1}=X_{i+1}$ and we deduce that $G = A^tX_{i+1}.$ This justifies the claim. For the remainder, without any loss of generality, we may assume that $M_1\cap \dots \cap M_t=H.$ 
			
			Now assume $j>t.$ Then either 
			\begin{itemize}\addtolength{\itemsep}{0.2\baselineskip}
\item[{\rm (a)}] $M_j=A^t K_j$ with $K_j$ a maximal subgroup of $H$; or 
\item[{\rm (b)}] $B_j=A^t \cap M_j \cong_H A^{t-1}.$ 
\end{itemize}			
Note that if (a) holds, then $H\cap M_j=K_j.$ On the other hand, if (b) holds then there exists a group $D_j\cong_H A$ and a derivation $\delta_j \in \Der(H,D_j)$ such that $A^t\cong_H B_j\times D_j$ and $M_j=B_j\{hh^{\delta_j} \,:\, h\in H\},$ so $K_j=H\cap M_j=C_{\delta_j}.$ 

We conclude that  
\[
1=H\cap M_{t+1} \cap \dots \cap M_\alpha=K_{t+1}\cap \dots \cap K_\alpha
\]
and therefore $\alpha-t\geqs \sigma$ as required.

Finally, set $b=b(H,A).$ Then there exist $a_1,\ldots,a_b \in A$ such that
$H, H^{a_1},\dots,H^{a_b}$ are maximal subgroups of $A\rtimes H$ and 
\[
H\cap H^{a_1}\cap\dots\cap H^{a_b}=C_H(a_1)\cap \dots \cap C_H(a_b)=1,
\] 
hence $\alpha(A \rtimes H)\leqs b+1.$
 As noted above, 
$\alpha \leqs t-1+\alpha(A \rtimes H),$ thus $\alpha \leqs t+b.$
Clearly $b\leqs \dim_{\End_H(A)}A,$ while \cite[Theorem 1.1]{ser} gives $b\leqs 3$ if $H$ is soluble. The result follows.
		\end{proof}
	
\begin{lem}\label{ab}
	Let $G$ be a finite group and let $A=H/K$ be a non-Frattini abelian chief factor of $G$. Then 
	\[
	\alpha(G/R_G(A))\leqs \delta_G(A)+\dim_{\End_G(A)}A.
	\]
	Moreover, if $G/C_G(A)$ is soluble, then 
	\[
	\alpha(G/R_G(A))\leqs \delta_G(A)+3.
	\]
\end{lem}	
	
	\begin{proof}
		Let $\delta=\delta_G(A).$ There exists an irreducible subgroup $J \leqs {\rm GL}(A)$ such that 
		\[
		G/R_G(A)\cong (A_1\times \dots\times A_\delta)\rtimes J,
		\]
		where  $A_i\cong_J A$ for every $i.$ The conclusion follows from Lemma \ref{l:ab0}.	\end{proof}

\begin{lem}\label{nonab}
	Let $G$ be a finite group and let $A=H/K$ be a non-abelian chief factor of $G$. Then 
	\[
	\alpha(G/R_G(A))\leqs \max\{4,\delta_G(A)\}+\left\lfloor\frac{3n_A-1}{2}\right\rfloor,
	\]
	 where $n_A$ is the number of composition factors of $A.$
\end{lem}

\begin{proof}
Let $\delta=\delta_G(A)$ and $L=L_A.$ We have to bound $\alpha(G/R_G(A))=\alpha(L_\delta).$ 

Write $N= \soc(L) \cong T_1 \times \cdots \times T_n\cong A$, where each $T_i$ is isomorphic to a fixed non-abelian simple group $T$. Let $\psi$ be the map from  $N_L(T_1) $ to ${\rm Aut}(T)$ induced by the conjugacy action on $T_1$.
Set $H=\psi(N_L(T_1))$ and note that $H$ is an almost simple group with socle
$T={\rm Inn}(T)=\psi(T_1)$. Let $R=\{r_1,\ldots,r_n\}$ be a right transversal of $N_L(T_1)$ in $L$. Then the homomorphism $\phi_R: L \to H \wr S_n$ given by
$$l \mapsto ( \psi(r_{1}^{} l r_{1 \pi_l}^{-1}), \ldots ,  \psi(r_n^{} l r_{n
	\pi_l}^{-1})) \pi_l$$
	is injective, where $\pi_l \in S_n$ satisfies $r_i^{}l r_{i \pi_l}^{-1} \in
N_L(T_1)$ for all $i$ and all $l \in L$. Therefore, we may identify $L$ with
its image in $H \wr S_n$. Under this identification, $N$ is contained in
the base subgroup $H^n$, while $T_i$ is a subgroup of the $i$-th
component of $H^n$. Note that $J=\{\pi_l \,:\, l\in L\}$ is a transitive subgroup of $S_n.$

First assume $\delta\geqs 2.$  For each $2\leqs i\leqs \delta$, let
$M_i=\{(l_1,\dots,l_{\delta})\in L_\delta \,:\, l_i=l_1\}.$ Then $M_2,\dots,M_{\delta}$ are maximal subgroups of $L_\delta$ with 
\[
X:=M_2\cap \dots \cap M_{\delta}=\diag(L_{\delta})=\{(l,\dots,l) \,:\, l \in L\}\cong L.
\]
Since every non-abelian finite simple group is $2$-generated, we may choose $a, b \in T$ such that $T=\langle a, b\rangle$. Set $\alpha=(a,\dots,a),$ $\beta=(b,\dots, b)$ in $N = T^n$ and consider 
\[
\begin{aligned}
Y_\alpha&=\{(l,l^\alpha,l_3,\dots,l_{\delta}) \,:\, l,l_3,\dots,l_{\delta} \in L\},\\ Y_\beta&=\{(l,l^\beta,l_3,\dots,l_{\delta}) \,:\,  l,l_3,\dots,l_{\delta} \in L\}.
\end{aligned}
\]
Notice that 
\[
X \cap Y_\alpha\cap Y_\beta=\{(y,\dots,y) \,:\, y \in C_L(\alpha,\beta)\}.
\]
Let $B=\{l \in L \,:\, \pi_l=1\}$ and let $(x_1,\dots,x_n)\in C_B(\alpha,\beta).$
Then 
\[
x_i\in C_{{\rm Aut}(T)}(a)\cap C_{{\rm Aut}(T)}(b)=C_{{\rm Aut}(T)}(T)=1
\]
for all $i$, so $X\cap Y_\alpha\cap Y_\beta\cap \diag(B^\delta)=1$ and thus $X\cap Y_{\a} \cap Y_{\b}$ is isomorphic to a subgroup of $J\leqs S_n.$ By the main theorem of \cite{3n2}, we have $\ell(J) \leqs \lfloor(3n-1)/2 \rfloor-1$, where $\ell(J)$ denotes the maximal length of a chain of subgroups in $J$ (with proper inclusions). Since $\frat(L_\delta)=1,$ it follows that $\alpha(L_\delta)\leqs \delta+\lfloor (3n-1)/2 \rfloor.$

To complete the proof, we may assume that $\delta=1$. Let $M$ be a maximal subgroup of $H$ such that $H = MT$, so $M_H=1$. The intersection $L \cap (M \wr J)$ is a maximal subgroup of $L$ (see \cite[Proposition 1.1.44]{BBE}), so by applying Theorem \ref{t:main1} we deduce that there exist maximal subgroups $M_1,M_2, M_3, M_4$ of $G$ of product type with 
\[
M_1\cap M_2 \cap M_3 \cap M_4 \cap H^n = 1.
\]
If we set $Y = M_1\cap M_2 \cap M_3 \cap M_4$, then $\ell(Y)\leqs \lfloor (3n-1)/2 \rfloor-1$ as before and we conclude that $\alpha(L)\leqs 4+\lfloor (3n-1)/2 \rfloor$.
\end{proof}

\begin{proof}[Proof of Theorem \ref{t:main4}]
Let $G$ be a finite group. Let $\mathcal B_{\text{ab}}$ (respectively $\mathcal B_{\text{nonab}}$) be the set of non-Frattini chief factors of $G$ that are $G$-equivalent to some abelian (respectively, non-abelian) minimal normal subgroup of $G/\frat(G).$ Then by Proposition \ref{general},
\[
\alpha(G)\leqs \sum_{A \in \mathcal B_{\text{ab}}\cup \mathcal B_{\text{nonab}}}\alpha(G/R_G(A))
\]
and thus the conclusion follows by combining Lemmas \ref{ab} and \ref{nonab}.
\end{proof}

\end{document}